\documentclass[reqno, 12pt, reqno]{amsart}

\usepackage{amsfonts, amsthm, amsmath, amssymb}
\usepackage{hyperref}
\hypersetup{colorlinks=false}
\usepackage{comment}
\usepackage[OT2, T1]{fontenc}
\usepackage{dsfont}

\DeclareSymbolFont{cyrletters}{OT2}{wncyr}{m}{n}
\DeclareMathSymbol{\Sha}{\mathalpha}{cyrletters}{"58}

\usepackage[all]{xy}
\usepackage[margin=3.4cm]{geometry}

\RequirePackage{mathrsfs}
\let\mathcal\mathscr

\numberwithin{equation}{section}

\newtheorem{theorem}{Theorem}[section]

\newtheorem{lemma}[theorem]{Lemma}

\newtheorem{corollary}[theorem]{Corollary}
\newtheorem{proposition}[theorem]{Proposition}

\theoremstyle{definition}

\newtheorem{remark}[theorem]{Remark}
\newtheorem{remarks}[theorem]{Remarks}

\newcommand{\BA}{{\mathbb {A}}}

\newcommand{\BF}{{\mathbb {F}}}
\newcommand{\BG}{{\mathbb {G}}}

\newcommand{\BL}{{\mathbb {L}}}

\newcommand{\BN}{{\mathbb {N}}}

\newcommand{\BP}{{\mathbb {P}}}
\newcommand{\BQ}{{\mathbb {Q}}}
\newcommand{\BR}{{\mathbb {R}}}

\newcommand{\BZ}{{\mathbb {Z}}}
\newcommand{\CA}{{\mathcal {A}}}
\newcommand{\CB}{{\mathcal {B}}}
\newcommand{\CC}{{\mathcal{C}}}
\newcommand{\CD}{{\mathcal {D}}}
\newcommand{\CE}{{\mathcal {E}}}

\newcommand{\CG}{{\mathcal {G}}}
\newcommand{\CH}{{\mathcal {H}}}
\newcommand{\CI}{{\mathcal {I}}}
\newcommand{\CJ}{{\mathcal {J}}}
\newcommand{\CK}{{\mathcal {K}}}
\newcommand{\CL}{{\mathcal {L}}}

\newcommand{\CN}{{\mathcal {N}}}

\newcommand{\CR}{{\mathcal {R}}}
\newcommand{\CS}{{\mathcal {S}}}

\newcommand{\CU}{{\mathcal {U}}}
\newcommand{\CV}{{\mathcal {V}}}
\newcommand{\CW}{{\mathcal {W}}}

\newcommand{\bb}{{\mathbf{b}}}
\newcommand{\cc}{{\mathbf{c}}}

\renewcommand{\phi}{\varphi}
\renewcommand{\rho}{\varrho}
\renewcommand{\epsilon}{\varepsilon}

\renewcommand{\leq}{\leqslant}

\renewcommand{\geq}{\geqslant}

\renewcommand{\bar}{\overline}

\newcommand{\bx}{\boldsymbol{x}}
\newcommand{\by}{\boldsymbol{y}}

\newcommand{\bt}{\boldsymbol{t}}

\newcommand{\bc}{\boldsymbol{c}}
\newcommand{\bd}{\boldsymbol{d}}
\newcommand{\bu}{\boldsymbol{u}}

\newcommand{\bv}{\boldsymbol{v}}

\newcommand{\bxi}{\boldsymbol{\xi}}
\newcommand{\blambda}{\boldsymbol{\lambda}}
\newcommand{\bLambda}{\boldsymbol{\Lambda}}
\newcommand{\bGamma}{\boldsymbol{\Gamma}}
\newcommand{\bsigma}{\boldsymbol{\sigma}}

\newcommand{\balpha}{\boldsymbol{\alpha}}

\newcommand{\e}{\textup{e}}
\newcommand{\aess}{\alpha_{\operatorname{ess}}}
\newcommand{\Pic}{\operatorname{Pic}}
\newcommand{\wtil}{\widetilde{w}}

\title[Strong approximation for quaternary quadratic forms]{Quantitative strong approximation for quaternary quadratic forms}
\date{\today}
\author{Zhizhong Huang \and Damaris Schindler \and Alec Shute}

\address{Institute of Mathematics, Academy of Mathematics and Systems Science, Chinese Academy of Sciences, Beijing, 100190, China}
\email{zhizhong.huang@yahoo.com}

\address{Mathematisches Institut, Georg-August-Universit\"{a}t G\"{o}ttingen, Bunsenstrasse 3--5,  37073, G\"{o}ttingen, Germany}
\email{damaris.schindler@mathematik.uni-goettingen.de}

\address{School of Mathematics, University of Bristol, Woodland Road, Bristol, UK}
\email{alec.shute@bristol.ac.uk}

\begin{document}
		\begin{abstract}
		The purpose of this article is twofold. On the one hand, we prove asymptotic formulas for the quantitative distribution of rational points on any smooth non-split projective quadratic surface. We obtain the optimal error term for the real place. On the other hand, we also study the growth of integral points on the three-dimensional punctured affine cone, as a quantitative version of strong approximation with Brauer--Manin obstruction for this quasi-affine variety.

	\end{abstract}
	\maketitle
	\tableofcontents
	
	\section{Introduction}
	Let $V\subseteq\BP_{\mathbb{Q}}^n$ be a smooth projective quadric defined by a non-degenerate integral quadratic form $F$ in $(n+1)$-variables. We always assume that $V$ has local points at all places of $\mathbb{Q}$. Let  $W\subseteq\BA^{n+1}$ be its corresponding affine cone. Let $\CV\subseteq \BP_{\BZ}^n,\CW\subseteq\BA_{\BZ}^{n+1}$ be respectively their integral models defined by $F$.	In \cite{PART1}, assuming $n\geqslant 4$, on employing the delta method in a form developed by Heath-Brown \cite{H-Bdelta}, we prove asymptotic formulas for the distribution of rational points on $V$, with optimal error terms. This is done via counting integral points of $\CW$ inside of an arbitrary adelic neighbourhood, with additional assumptions regarding the real condition.  From this we deduce the best growth rate of the ``size'' of local conditions for which equidistribution is preserved.
		
	Adopting the approach from \cite{PART1}, the goal of the present concurrent article is to focus on the non-split quaternary case, i.e. we assume that $n=3$ and $\Delta_F$, the discriminant of $F$, is not a square.
    In this case the delta method has already been applied by Lindqvist \cite{Lindqvist}  in the context of strong approximation for the affine cone and she found additional main terms which did not have any clear interpretation yet. We provide a generalisation of her work to quantitative strong approximation, we compute the additional appearing terms explicitly and we show how they are related to the Brauer-Manin obstruction. 
    This is the first instance that we are aware of in which the circle method detects the Brauer-Manin obstruction in a quantitative way with an explicit leading constant. 
	
	\subsection{Effective equidistribution on the non-split projective quadratic surface}
	Let us fix a parameter space $S$ and a family of smooth weight functions $\mathscr{C}(S)$ as in \cite{H-Bdelta}. As in \cite[(1.1)]{PART1}, we define a \textit{real zoom condition} in terms of a fixed real point $\bxi\in V(\BR)$, a neighborhood $U_0$ of $\bxi$, and a scaling parameter $R \geq 1$. For $B\gg 1$, we choose $w_{B,R}:\BR^{4}\to \BR$ (see \eqref{eq:weightBR}), constructed from a certain fixed $w\in \mathscr{C}(S)$ as in \cite{PART1}, which describes the bounded height condition depending on $B$ and the real zoom condition. It satisfies \begin{equation}\label{eq:symmetric}
		w_{B,R}(-\bx)=w_{B,R}(\bx).
	\end{equation} See \S\ref{se:poisson} for a sketch and \cite[\S2]{PART1} for more details. 
	By abuse of notation, for $P\in \BP^3(\BQ)$, $w_{B,R}(P)$ means that we choose a lift of $P$ into a primitive integer vector of $\BZ^4$ and apply $w_{B,R}$.

	Our first result is on equidistribution with only a real zoom condition, and provides an \emph{optimal} error term.  
	We consider the counting function 
	\begin{equation}\label{eq:countingfunVreal}
		\CN_{V}(w_{B,R}):=\sum_{\substack{P\in V(\BQ)}}w_{B,R}(P).
	\end{equation} 
 
	\begin{theorem}\label{thm:realzoommain}
			Assume \eqref{eq:symmetric}, $\Delta_F\neq\square$, and \begin{equation}\label{eq:Rtau}
				R^2\asymp B^{1-\tau} \quad \text{for some } 0<\tau<1.
			\end{equation} 
   Assume that $V$ has local points at all places of $\mathbb{Q}$. Then
		$$\CN_{V}(w_{B,R})=B^{2}\frac{\CI(w_R)}{2}\mathfrak{S}(V)\left(1+O_\varepsilon\left(B^{-\tau+\varepsilon}\right)\right),$$ where $\CI(w_R)$ is the ``weighted singular integral'' (see \eqref{eq:weightedsingint} for its definition) and $\mathfrak{S}(V)$ is the finite Tamagawa measure of $\CV(\widehat{\BZ})$ (à la Peyre \cite{Peyre}, with respect to the naive metric). 
	\end{theorem}

	 For $L\in\BN$, we let  $\bLambda=(\bLambda_p)_{p\mid L}$, where $\bLambda_p\in \CV(\BZ_p)$ for every $p\mid L$. The pair $(L,\bLambda)$ defines a finite adelic open subset $\CD(L,\bLambda)\subset\CV(\widehat{\BZ}):=\prod_p \CV(\BZ_p)$ (see \cite{PART1}), and gives rise to a \textit{$p$-adic zoom condition} for each $p \mid L$. We now describe our second result of counting with only $p$-adic zoom conditions (i.e. we take $R=1$).
	 We shall write $$w_B=w_{B,1}$$ and consider 	\begin{equation}\label{eq:countingfunVpadic}
	 	\CN_{V}(w_{B};(L,\boldsymbol{\Lambda})):=\sum_{\substack{P\in V(\BQ)\\ P\in \CD(L,\bLambda)}}w_{B}(P).
	 \end{equation} 
	\begin{theorem}\label{thm:padiczoommain}
		Assume \eqref{eq:symmetric}, $\Delta_F\neq\square$, and
		\begin{equation}\label{eq:tauLdelta}
	    L^2\asymp B^{1-\tau}	\text{ for some }	\frac{1}{2}<\tau< 1.
		\end{equation} Assume that $V$ has local points at all places of $\mathbb{Q}$. Then we have
		$$\CN_{V}(w_B;(L,\bLambda))= B^2\frac{\CI(w)}{2}	\mathfrak{S}_{L,\bLambda}(\CV)\left(1+O_\varepsilon\left(B^{-2\tau+1+\varepsilon}\right)\right),$$
		where $\mathfrak{S}_{L,\bLambda}(\CV)$ is the finite Tamagawa measure of $\CV(\widehat{\BZ})$ regarding the local condition $(L,\bLambda)$. 
	\end{theorem}

\begin{remark}
		Theorems \ref{thm:realzoommain} and  \ref{thm:padiczoommain} have the following consequence on 
	the ``\emph{essential approximation constant}'' $\aess^{\nu}$ (see \cite[Appendix A]{PART1} for more details): For any real point $\boldsymbol{\xi}\in V(\BR)$, we have
	$$\aess^{\BR}(\boldsymbol{\xi})\leqslant 2;$$ 	For any non-archimedean place $\nu$ and any $\nu$-adic point $\xi_\nu\in V(\BQ_\nu)$, we have $$\aess^{\nu}(\xi_\nu)\leqslant 4.$$ 
	Note that in \cite[Appendix A]{PART1} we show that $\aess^{\nu}(\xi)=2$ when $\xi\in V(\mathbb{Q})$. So Theorem \ref{thm:realzoommain} is optimal if $\bxi$ is a rational point. It would be very interesting to determine the exact value of such an essential constant for non-archimedean places. This is related to the best error term one can obtain in Theorem \ref{thm:padiczoommain}.
\end{remark}

	\subsection{Counting integral points on the punctured affine cone}
	The second goal of this article is to study the asymptotic growth of integral points on the punctured affine cone $\CW^o:=\CW\setminus\overline{\boldsymbol{0}}\subseteq\BA^4_{\BZ}$. In stark contrast to the higher dimensional cases, under the assumption that $\Delta_F\neq\square$, there can be a Brauer--Manin obstruction to strong approximation on $\CW^o$. 
	
	On $\CW^o(\widehat{\BZ}):=\prod_p \CW^o(\BZ_p)$, we can associate a (normalised) Tamagawa measure, denoted by $\omega_f$ locally given by the ``modified singular series'' (see \eqref{eq:omegaf} and see also Appendix A). Let $W_i,i\in I$ be the connected components of the real locus $W^o(\BR)$. If $\Delta<0$ (resp. $\Delta>0$) we then have $\#I=2$ (resp. $\# I=1$). (See Lemma \ref{lemcomp}.)
    
	We shall establish in \S\ref{se:BMCK} the following result.
	\begin{theorem}\label{thm:HLWo}
		Assume $\Delta_F\neq\square$. Then for every $i\in I$, there exists a locally constant function $\Xi_{W_i}:\CW^o(\widehat{\BZ})\to \{0,2\}$ such that for any  non-empty  open compact $\CE\subseteq \CW^o(\widehat{\BZ})$, if the support of $w_B$ is contained in  $W_i$, we have as $B\to\infty$,	$$\sum_{\bx\in \CW^o(\BZ)\cap (W_i\times\CE)} w_B(\bx)\sim \CI(w)\left(\int_{\CE}\Xi_{W_i} \operatorname{d}\omega_f\right) B^2.$$ 
        
	\end{theorem}
Theorem \ref{thm:HLWo} implies a quantitative version of the Brauer--Manin obstruction to strong approximation of integral points on $\CW^o$.
According to the authors’ knowledge, this is the first instance where the circle method furnishes asymptotic formulas that encode information from the Brauer--Manin obstruction.\par
Note that Theorem \ref{thm:HLWo} also encodes a failure of strong approximation due to a Brauer-Manin obstruction in cases where the locally constant function $\Xi_{W_i}$ is equal to zero. This shows that an analogue of Corollary 3 in \cite{H-Bdelta} for primitive integer vectors in a strict sense is not true due to the possible existence of a Brauer-Manin obstruction (compare also with the paragraph after Corollary 3 in \cite{H-Bdelta} where it is claimed that a generalisation would be possible.)

\begin{remarks}
	\hfill\begin{enumerate}
		\item 	 In the terminology of Borovoi--Rudnick \cite{Borovoi-Rudnick}, up to the weighted singular integral, (an extension of) Theorem \ref{thm:HLWo} shows that (see also \cite[Proposition 2.4]{Cao-Huang2}):
		\emph{The quasi-affine variety $W^o$ is relatively Hardy--Littlewood (\cite[p. 39]{Borovoi-Rudnick}) with density function $\Xi_{W_i}$ for each connected component $W_i$.}
		
		\item 	Applying results on the geometric sieve for quadrics (\cite[Theorem 1.3]{Browning-HB} or a variant of \cite[Theorem 1.6]{Cao-Huang2}), the following stronger equidistribution result holds (see \cite[Definition 1.2]{Cao-Huang2}):
		\emph{The quasi-affine variety $W^o$ satisfies arithmetic purity of the Hardy--Littlewood property.}
		Equivalently, every open subset of $W^o$ whose complement has codimension at least two is also relatively Hardy--Littlewood.
	\end{enumerate}

\end{remarks}

\subsection{Methods and further comments}

Heath-Brown \cite{H-Bdelta} pioneers the counting for quaternary quadratic forms, applying his $\delta$-method. Our investigation also makes use of the $\delta$-method, however, to minimize the error terms in Theorems \ref{thm:realzoommain} and  \ref{thm:padiczoommain}, we need to refine a number of arguments in \cite{H-Bdelta,PART1}, and in particular, a ``double Kloosterman refinement'' is brought into play.  The  finite Tamagawa measure on $\CV(\widehat{\BZ})$ features a different form of Eulerian product, due to the necessity of introducing convergence factors. 
We also prove the accordance of the leading constant with Peyre's prediction (see Proposition \ref{prop:TamagawaLLambda}), which does not seem to have been addressed before in the literature.

Work of Lindqvist \cite{Lindqvist} first addresses counting with extra congruence conditions. One has to be more careful when handling quadratic exponential sums as the multiplicativity is lost. In \cite{Lindqvist} it is observed that, after executing the Poisson summation and introducing the Poisson variable $\bc\in\BZ^4$, the leading term does not always agree with the contribution from $\bc=\boldsymbol{0}$. A mysterious main term $\CK$ (see \eqref{eq:CKc}) appears which is expressed as an infinite sum over certain $\bc\neq \boldsymbol{0}$ with $F^*(\bc)=0$.  We are able to completely determine the value of $\CK$, and show that its appearance should be thought of as the ``defect'' of strong approximation (see Remark \ref{rmk:measureint}), thereby completing the work \cite{Lindqvist}.

We have chosen to establish the main theorems only with respect to a specific choice of smooth weights. This is solely for technical simplicity. The same procedure of ``removing smooth weights'' can be applied as in \cite[\S7.2]{PART1}, which would only affect the size of the power savings in Theorems \ref{thm:realzoommain} and \ref{thm:padiczoommain} and not the allowable growth rate of $R$ and $L$ for which the asymptotic formulas are obtained. Obtaining a joint version with both real and finite local conditions as in \cite[Theorem 1.1]{PART1} is also possible.

	\subsection{Notation and conventions} Throughout this article we assume $n=3$ and $\Delta_F\neq\square$. We assume that the quadratic form $F$ has even non-diagonal terms (which we can do after scaling), in particular we then have $\Delta_F\in \mathbb{Z}$, where $\Delta_F$ is the determinant of the corresponding matrix defining $F$.
	
	The symbol $\chi$ denotes a Dirichlet character, the modulus of which is denoted by $|\chi|$. We write $\chi^c$ for its conjugate, so as not to conflict with the notation of multiplicative inverse. We write $$\BL(s,\chi):=\sum_{n=1}^{\infty}\frac{\chi(n)}{n^s}$$ for the associated Dirichlet series.
	For $n\in\BN$, $\chi_0[n]$ denotes the principal Dirichlet character modulo $n$.
	We define the (non-principal) Dirichlet character $\psi_F(\cdot):=\left(\frac{4\Delta_F}{\cdot}\right)$ and write $\widetilde{\psi_F}$ for the primitive character which induces $\psi_F$.
	As usual we write $\e(x):=\exp\left(2\pi i x\right)$ and $\e_q(x):=\exp\left(\frac{2\pi i x}{q}\right)$. Let $\phi$ be the Euler totient function.
	We define the arithmetic functions
	\begin{equation}\label{eq:theta}
		\theta_1(n):=\prod_{p\mid n}\left(1-\frac{1}{p}\right);\quad 
	 	\theta_2(n):=\theta_1(n)\prod_{p\nmid n}\left(1-\frac{1}{p^2}\right).
	 \end{equation}
	 We have $\phi(n)=n\theta_1(n)$.

		\section{Set-up}
		
	Recall the following identity about the $\delta$-symbol in Heath-Brown's work \cite[Theorem 1]{H-Bdelta}: Let $n$ be an integer, and define $\delta(n)$ to be $1$ if $n=0$ and zero otherwise. Then there exists a ``nice'' function $h:\BR_{>0}\times \BR\to\BR$ such that for every $Q>1$, we have $$\delta(n)=\frac{C_Q}{Q^2}\sum_{q=1}^{\infty}\sum_{\substack{a\bmod q\\(a,q)=1}}\textup{e}_q\left(an\right)h\left(\frac{q}{Q},\frac{n}{Q^2}\right)$$ for a certain $C_Q=1+O_N(Q^{-N})$.
	
	\subsection{The congruence condition and Poisson summation}\label{se:poisson}
		For $w:\BR^4\to\BR_{\geqslant 0}$ a smooth weight function of class $\CC(S)$, and for $L\in\BN$ and $\bGamma\in \CW^o(\BZ/L\BZ)$,
		we consider the counting function (with respect to $\CW$)
		\begin{equation}\label{eq:countingW}
			\CN_{\CW}(w;(L,\bGamma)):=\sum_{\substack{\bx\in\BZ^4: F(\bx)=0\\\bx\equiv \bGamma\bmod L}}w(\bx).
		\end{equation}
		
		We choose $\blambda=\blambda(\bGamma)\in \BZ^4$ a lift of $\bGamma$ such that $|\blambda|\leqslant L$ and $\gcd(L,\lambda_0,\lambda_1,\lambda_2,\lambda_3)=1$. Define the affine linear polynomial 
		$$H_{\boldsymbol{\lambda},L}(\by):=\frac{F(\boldsymbol{\lambda})}{L}+\nabla F(\boldsymbol{\lambda})\cdot\by.$$
		We observe that for $\by\in\BZ^4$,
		\begin{equation}\label{eq:equivdiv}
			H_{\boldsymbol{\lambda},L}(\by)\equiv 0\bmod L\Leftrightarrow F(L\by+\boldsymbol{\lambda})\equiv 0\bmod L^2.
		\end{equation}
Then a standard manipulation via Poisson summation as in \cite[\S2.4]{PART1} yields
	\begin{equation}\label{eq:poisson}
		\begin{split}
		\CN_{\CW}(w;(L,\bGamma))&=\sum_{\substack{\by\in\BZ^4\\ L\mid H_{\boldsymbol{\lambda},L}(\by)}} w(L\by+\boldsymbol{\lambda})\delta\left(\frac{F(L\by+\boldsymbol{\lambda})}{L^2}\right)\\ &=\frac{C_Q}{Q^2}\sum_{q=1}^{\infty}\sum_{\bc\in\BZ^4}\frac{S_{q,L,\blambda}(\bc)I_{q,L,\blambda}(w;\bc)}{(qL)^4},		\end{split}
	\end{equation} where for every $q,\bc$, we define the \emph{weighted oscillatory integral}
	\begin{equation}\label{eq:Iqc1}
		I_{q,L,\blambda}(w;\bc):=\int_{\BR^4}w(L\by+\boldsymbol{\lambda})h\left(\frac{q}{Q},\frac{F(L\by+\boldsymbol{\lambda})}{L^2Q^2}\right)\e_{qL}\left(-\bc\cdot\by\right)\operatorname{d}\by,
	\end{equation} and the \emph{twisted quadratic exponential sums} \begin{equation}\label{eq:Sqc}
		S_{q,L,\blambda}(\bc):=\sum_{\substack{a\bmod q\\(a,q)=1}}\sum_{\substack{\bsigma\in(\BZ/qL\BZ)^4\\ H_{\boldsymbol{\lambda},L}(\boldsymbol{\sigma})\equiv \boldsymbol{0}\bmod L}}\textup{e}_{qL}\left(a\left(H_{\boldsymbol{\lambda},L}(\boldsymbol{\sigma})+LF(\boldsymbol{\sigma})\right)+\bc\cdot\bsigma\right).
	\end{equation}

\subsection{The real condition and the oscillatory integrals}

We pause to describe briefly the precise form of the weight function which we shall use.
Upon fixing a real point $\bxi\in V(\BR)$, by means of Witt's cancellation theorem, we do a linear change of variables $\bx\mapsto \bt$ using a matrix $M\in\operatorname{GL}_{4}(\BR)$ such that $\bxi\mapsto[1:0:0:0]$ and \begin{equation}\label{eq:FWitt}
	\widetilde{F}(\bt):=F(M^{-1}\bt)=t_0t_1 + F_2(t_2,t_3),
\end{equation} where $F_2$ is a certain non-degenerate binary quadratic form. Then $\left(\frac{t_2}{t_0}(\bx),\frac{t_3}{t_0}(\bx)\right)$ forms a local coordinate system of $V(\BR)$ around $\bxi$. Taking the real zoom condition above into account, we therefore let   
\begin{equation}\label{eq:weightBR}
	w_{B,R}(\bx)=w_1\left(R\left(\frac{t_2}{t_0}(\bx),\frac{t_3}{t_0}(\bx)\right)\right)w_2\left(\frac{R^2}{B^2}F(\bx)\right)w_3\left(\frac{t_0(\bx)}{B}\right),
\end{equation} for certain fixed $w_1:\BR^2\to\BR;w_2,w_3:\BR\to\BR$, all of which are of class $\mathscr{C}(S)$ and $w_3$ satisfying $w_3(-x)=w_3(x)$ as well as $w_3(0)\neq 0$.
Therefore $w_{B,R}(\bx)=w_{B,R}(-\bx)$.

We then choose \begin{equation}\label{eq:Q}
	Q:=\frac{B}{LR}
\end{equation} as in \cite[\S2.5]{PART1}.

We next execute a preliminary manipulation for $I_{q,L,\blambda}(w_{B,R};\bc)$, as is done in \cite[\S3.1]{PART1}:
	\begin{equation}\label{eq:IJ}
		I_{q,L,\blambda}(w_{B,R};\bc)=\frac{\e_{qL^2}(\bc\cdot\blambda)}{|\det M|}\left(\frac{B}{LR}\right)^4 \widehat{\CI}_{q}(\widehat{w};\boldsymbol{d}),
	\end{equation} 
where
\begin{equation}\label{eq:CIq}
	\widehat{\CI}_{q}(\widehat{w};\boldsymbol{d}):=\int_{\BR^4} \widehat{w}(\bu)h\left(\frac{q}{Q},\widetilde{F}(\bu)\right)\e_q\left(-\bd\cdot\bu\right)\operatorname{d}\bu,
\end{equation} where
\begin{equation}\label{eq:hatw}
	\widehat{w}(\bu):=w_1\left(\frac{u_2}{u_0},\frac{u_3}{u_0}\right)w_2\left(\widetilde{F}(\bu)\right)w_3\left(u_0\right)
\end{equation}
depending only on the ``constant part'' $w_1,w_2,w_3$ of $w_{B,R}$, and we associate the vector $\boldsymbol{d}=\boldsymbol{d}(B,R,L;\cc)\in\BR^4$ to $\bc$ by \begin{equation}\label{eq:d}
	d_0=\frac{B}{L^2}\widetilde{c_0},\quad d_1=\frac{B}{L^2R^2}\widetilde{c_1},\quad d_2=\frac{B}{L^2R}\widetilde{c_2},\quad d_3=\frac{B}{L^2R}\widetilde{c_3},
\end{equation} with \begin{equation}\label{eq:tildec0}
\widetilde{\bc}=(\widetilde{c_0},\cdots,\widetilde{c_3}):=(M^{-1})^t\bc\in\BR^4.
\end{equation} 

When the real condition is trivial (i.e. $R=1$), we just let $M$ be the identity.\par
In order to see the effect of the  Brauer--Manin obstruction at the real place, we shall work with a slightly different shape of weight functions. Let $w_0: \mathbb{R}^4\rightarrow \mathbb{R}$ be a smooth compactly supported weight function such that the support of $w_0$ is contained in a connected component of $W^o(\mathbb{R})$. We define $\wtil_B: \mathbb{R}^4\rightarrow \mathbb{R}$ by
 \begin{equation}\label{eq:weightpm}
	\wtil_B(\bx):=w_0\left(\frac{\bx}{B}\right).
\end{equation} 
The change of variables $\bu=\frac{L\by+\blambda}{B}$ yields
$$I_{q,L,\blambda}(\wtil_B;\bc)=\e_{qL^2}(\bc\cdot\blambda)\widehat{\CI}_{q}(w_0;\boldsymbol{d})\left(\frac{B}{L}\right)^4,$$ where $R=1$, $M$ is the identity matrix and $\boldsymbol{d}$ is defined as before in \eqref{eq:d}.\par
We next record Heath-Brown's ``simple estimates'' and ``harder estimates'' \cite[\S7\S8]{H-Bdelta}. We abbreviate in what follows all the functions \eqref{eq:hatw} and $w_0$ by $\widehat{w}$. For $\bc=\boldsymbol{0}$ we have
\begin{lemma}[\cite{H-Bdelta} Lemma 16]\label{le:Ic=0}
	We have $$\frac{\partial\widehat{\CI}_{q}(\widehat{w};\boldsymbol{0})}{\partial q}\ll \frac{1}{q}\quad \mbox{ and }\quad \widehat{\CI}_{q}(\widehat{w};\boldsymbol{0})\ll 1.$$
\end{lemma} 
Concerning $\bc\neq\boldsymbol{0}$ (see also \cite[Corollary 3.3]{PART1}), we have
\begin{lemma}[\cite{H-Bdelta} Lemma 18]\label{le:easy}
	Assume $\bd\neq\boldsymbol{0}$. Then for any $N>0$,
	$$\widehat{\CI}_{q}(\widehat{w};\boldsymbol{d})\ll_N \frac{Q}{q}\left(\frac{|\bd|}{Q}\right)^{-N}.$$
\end{lemma}
\begin{lemma}[\cite{H-Bdelta} Lemma 22]\label{le:hard}
	Assume $\bd\neq\boldsymbol{0}$. Then $$\widehat{\CI}_{q}(\widehat{w};\boldsymbol{d})\ll_\varepsilon q \frac{Q^\varepsilon}{|\bd|^{1-\varepsilon}},$$ and
	$$\frac{\partial \widehat{\CI}_{q}(\widehat{w};\boldsymbol{d})}{\partial q}\ll_\varepsilon \frac{Q^\varepsilon }{|\bd|^{1-\varepsilon}}.$$
\end{lemma}
\begin{proof}[Sketch of Proof]
	In the notation of \cite[\S7]{H-Bdelta} $$\widehat{\CI}_{q}(\widehat{w};\boldsymbol{d})=I_{r}^*\left(\frac{\bd}{Q}\right)$$ for $r:=\frac{q}{Q}$. 
	Then Lemma \ref{le:easy} follows directly from \cite[Lemmata 14\&18]{H-Bdelta}. On the other hand, \cite[Lemma 14]{H-Bdelta} also gives
	$$\frac{\partial \widehat{\CI}_{q}(\widehat{w};\boldsymbol{d})}{\partial q}=\frac{1}{Q}\frac{\partial I_{r}^*\left(\frac{\bd}{Q}\right)}{\partial r}\ll \frac{Q}{q^2}\left|\int_{\BR^4}\widehat{w}_1(\bt)f_{\frac{q}{Q}}(\widetilde{F}(\bt))\e\left(-\frac{\bd}{q}\cdot \bt\right)\operatorname{d}\bt\right|\footnote{In \cite{H-Bdelta} this integral is denoted by $I(r;\bu)$.},$$ where the function $f_s(y)$ is either $sh(s,y)$ or $s^2\frac{\partial h(s,y)}{\partial s}$ which all belong to  the class $\CH$ of \cite[(7.1)]{H-Bdelta}, and  $\widehat{w}_1\in \CC_0(S)$ is a certain weight function.
	So Lemma \ref{le:Ic=0} follows from \cite[Lemma 16]{H-Bdelta} and Lemma \ref{le:hard} follows from  \cite[Lemma 22]{H-Bdelta}. 
\end{proof}
\subsection{The quadratic exponential sums}
	Following \cite[\S4.1]{PART1}, let us decompose  \begin{equation}\label{eq:qdecomp}
		q=q_1q_2, \quad \frac{F(\blambda)}{L}=k_2q_1+k_1q_2L
	\end{equation} with $\gcd(q_1,q_2L)=1$ and $k_1\bmod q_1,k_2\bmod q_2L$.
We use the Chinese remainder theorem to decompose
 \begin{equation}\label{eq:sqdecomp}
	S_{q,L,\blambda}(\bc)=S^{(1)}_{q,L,\blambda}(\bc)S^{(2)}_{q,L,\blambda}(\bc),
\end{equation}
where we define \begin{equation}\label{eq:S1}
		\begin{split}
			S^{(1)}_{q,L,\blambda}(\bc)&:=\sum_{\bsigma_1\in (\BZ/q_1\BZ)^4}\sum_{\substack{a_1\bmod q_1\\(a_1,q_1)=1}}T^{(1)}_{q,L,\blambda}(a_1,\bsigma_1,\bc),
		\end{split}
	\end{equation}  with
	$$T^{(1)}_{q,L,\blambda}(a_1,\bsigma_1,\bc):=\e_{q_1}\left(a_1\left((q_2L)^2F(\bsigma_1)+q_2(\nabla F(\blambda)\cdot \bsigma_1+k_1)\right)+\bc\cdot\bsigma_1\right),$$ and we also let \begin{equation}\label{eq:S2}
		S^{(2)}_{q,L,\blambda}(\bc):=\sum_{\substack{\bsigma_2\in (\BZ/q_2L\BZ)^4\\ H_{\blambda,L}(q_1\bsigma_2)\equiv 0\bmod L}}\sum_{\substack{a_2\bmod q_2\\(a_2,q_2)=1}}T^{(2)}_{q,L,\blambda}(a_2,\bsigma_2,\bc),
	\end{equation}  with
	$$T^{(2)}_{q,L,\blambda}(a_2,\bsigma_2,\bc):=\e_{q_2L}\left(a_2\left(q_1^2L F(\bsigma_2)+q_1(\nabla F(\blambda)\cdot \bsigma_2+k_2)\right)+\bc\cdot\bsigma_2\right).$$

	If moreover $\gcd(q_1,2L\Delta_F)=1$, by \cite[Proposition 4.4]{PART1} the $S^{(1)}$-term can be evaluated:
	\begin{equation}\label{eq:S1R}
		S^{(1)}_{q,L,\blambda}(\bc)=\e_{q_1}\left(-\overline{q_2L^2}\bc\cdot\blambda\right) \CR(q_1;\bc)
	\end{equation}where \begin{equation}\label{eq:Rqc}
	\CR(q_1;\bc):=q_1^2\psi_F(q_1)\sum_{\substack{a\bmod q_1\\(a,q_1)=1}}\e_{q_1}\left(-aF^*(\bc)\right).
\end{equation}

	As for the $S^{(2)}$-term, in \cite[Proposition 4.6]{PART1} we have established
	\begin{equation}\label{eq:S2bd}
			S^{(2)}_{q,L,\blambda}(\bc)\ll q_2^{3}L^3\gcd(q_2,L)^\frac{1}{2}.
	\end{equation}
The following improves upon \cite[Theorem 4.1]{PART1} in the case $n=3$.
\begin{proposition}\label{prop:fstarcneq0}	
		Uniformly for any $q,L,\blambda$ and $\bc$ we have
		$$S_{q,L,\blambda}(\bc)\ll (qL)^3\gcd(q,L)^\frac{1}{2}.$$
		If moreover $F^*(\bc)\neq 0$, we have  $$\sum_{q\leqslant X}\left|S_{q,L,\blambda}(\bc)\right|\ll_\varepsilon X^{3+\varepsilon}|\bc|^\varepsilon L^{3+\varepsilon}\min(X,L)^\frac{1}{2}.$$
\end{proposition}
\begin{proof}
		The uniform bound follows directly from \eqref{eq:S1R} \eqref{eq:Rqc} \eqref{eq:S2bd}.
		
	Assume $F^*(\bc)\neq 0$. On recalling \eqref{eq:Rqc}, the sum inside of $\CR(q_1;\bc)$ is a Ramanujan sum, which can be controlled  by 
	$$\left|\sum_{\substack{a\bmod q_1\\(a,q_1)=1}}\e_{q_1}\left(-aF^*(\bc)\right)\right|\leqslant \gcd(q_1,F^*(\bc)),$$ so that $$\left|S^{(1)}_{q,L,\blambda}(\bc)\right|\leqslant q_1^2\gcd(q_1,F^*(\bc)).$$ Compared to \cite[Corollary 4.5]{PART1}, this improved estimate for $S^{(1)}$ saves the factor $q_1^\frac{1}{2}$, while losing $\gcd(q_1,F^*(\bc))^\frac{1}{2}$. This is crucial in optimising the bound in the $X$-aspect.
	Thanks to \eqref{eq:sqdecomp} and using \eqref{eq:S2bd}, we have
	\begin{align*}
		\sum_{q\leqslant X}\left|S_{q,L,\blambda}(\bc)\right|&=\sum_{\substack{q_1q_2\leqslant X\\ (q_1,2L\Delta_F)=1,q_2\mid (2L\Delta_F)^\infty}}\left|S^{(1)}_{q,L,\blambda}(\bc)S^{(2)}_{q,L,\blambda}(\bc)\right|\\ &\leqslant  \sum_{\substack{q_2\leqslant X\\ q_2\mid (2L\Delta_F)^\infty}}q_2^3 L^3\gcd(q_2,L)^\frac{1}{2}\sum_{\substack{q_1\leqslant \frac{X}{q_2}\\ (q_1,2L\Delta_F)=1}}q_1^2\gcd(q_1,F^*(\bc))\\ &\ll |\bc|^\varepsilon L^3\sum_{\substack{q_2\leqslant X\\ q_2\mid (2L\Delta_F)^\infty}}q_2^3\gcd(q_2,L)^\frac{1}{2}\left(\frac{X}{q_2}\right)^3\\ &\ll_\varepsilon X^3|\bc|^\varepsilon L^3\sum_{\substack{q_2\leqslant X\\ q_2\mid (2L\Delta_F)^\infty}}\min(q_2,L)^\frac{1}{2}\\ &\ll_\varepsilon X^{3+\varepsilon}|\bc|^\varepsilon L^{3+\varepsilon}\min(X,L)^\frac{1}{2}.
	\end{align*}
The proof is completed.
\end{proof}

Recall the function $\theta_1$ from \eqref{eq:theta}. If $F^*(\bc)=0$, it follows from \eqref{eq:Rqc} that we have
\begin{equation}\label{eq:RqcFstarc=0}
	\CR(q;\bc)=q^3\theta_1(q)\psi_F(q)=\phi(q^3)\psi_F(q).
\end{equation}
The following elementary estimate will be frequently used.
\begin{lemma}\label{prop:fstarc=0}
	Let $\chi$ be a Dirichlet character of modulus $|\chi|$.  If $\chi\psi_F$   is not principal,  then
	$$\sum_{\substack{n\leqslant X}}\chi(n)\psi_F(n)\theta_1(n)=O_\varepsilon\left(X^{\varepsilon}|\chi|^{\frac{1}{2}+\varepsilon}\right).$$ 
\end{lemma}
\begin{proof}
	We have
	\begin{align*}
		\sum_{\substack{n\leqslant X}}\chi(n)\psi_F(n)\theta_1(n)&=\sum_{e_1\leqslant X}\chi(e_1)\psi_F(e_1)\frac{\mu(e_1)}{e_1}\sum_{e_2\leqslant \frac{X}{e_1}}\chi(e_2)\psi_F(e_2)\ll_\varepsilon X^{\varepsilon}|\chi|^{\frac{1}{2}+\varepsilon},
	\end{align*} on applying the Polya--Vinogradov bound to the inner character sum. \end{proof}

\subsection{The contribution from $\bc=\boldsymbol{0}$}
We define a ``modified'' singular series
\begin{equation}\label{eq:singser}
	\widetilde{\mathfrak{S}}_{L,\bGamma}(\CW):=\BL(1,\psi_F)\prod_{p<\infty}\left(1-\frac{\psi_F(p)}{p}\right)\sigma_p(\CW;L,\bGamma),
\end{equation}
where $$\sigma_p(\CW;L,\bGamma):=\lim_{k\to\infty}\frac{\#\{\bv\in \CW(\BZ/p^k\BZ):\bv\equiv\bGamma\bmod p^{\operatorname{ord}_p(L)}\}}{p^{3k}},$$ is the usual $p$-adic density of $\CW$, but with an extra congruence condition $\bv\equiv\bGamma\bmod p^{\operatorname{ord}_p(L)}$ included to account for the $p$-adic zoom condition. 
\begin{proposition}\label{prop:c=0L}
	Let $L,R\geqslant 1$. Let $w_{B,R}$ be either given by \eqref{eq:weightBR} or \eqref{eq:weightpm}. We have
	$$\sum_{q \ll Q}\frac{S_{q,L,\blambda}(\boldsymbol{0})I_{q,L,\blambda}(w_{B,R};\boldsymbol{0})}{(qL)^4}=\frac{B^4}{(LR)^2}\left(\CI(w_R)\widetilde{\mathfrak{S}}_{L,\bGamma}(\CW)+O_\varepsilon\left(B^{\varepsilon}(LR)^{-2}Q^{-1}\right)\right),$$ where for the case \eqref{eq:weightBR},	\begin{equation}\label{eq:weightedsingint}
		\CI(w_R):=
			\int_{\BR^4}\int_{\BR} w_1\left(R\left(\frac{t_2}{t_0}(\bx),\frac{t_3}{t_0}(\bx)\right)\right) w_2(F(\bx))w_3(t_0(\bx))\e\left(\theta F(\bx)\right)\operatorname{d}\bx\operatorname{d}\theta,
	\end{equation} and for the case \eqref{eq:weightpm}, 
	\begin{equation}\label{eq:weightedsingintpm}
		\CI(w_R)=\CI(w_0):=\int_{\BR^4}\int_{\BR}w_0(\bx)\e\left(\theta F(\bx)\right)\operatorname{d}\bx\operatorname{d}\theta.
	\end{equation}
\end{proposition}
\begin{proof}
	By \eqref{eq:S1R} and \eqref{eq:Rqc}, 
	$$S^{(1)}_{q,L,\blambda}(\boldsymbol{0})=S_{q_1}(\boldsymbol{0}):=q_1^3\theta_1(q_1)\psi_F(q_1),$$ which is independent of $q_2$. By Lemma \ref{prop:fstarc=0},
	$$\sum_{\substack{q\leqslant Y\\ \gcd(q,2L\Delta_F)=1}}\frac{S_{q}(\boldsymbol{0})}{q^3}=\sum_{n\leqslant Y}\chi_0[2L\Delta_F](n)\psi_F(n)\theta_1(n)\ll_\varepsilon L^{\frac{1}{2}+\varepsilon} Y^\varepsilon.$$
	Moreover, by \eqref{eq:S2}, the change of variables $\bsigma_2\mapsto \bsigma_2':=q_1\bsigma_2$ yields
	$$S^{(2)}_{q,L,\blambda}(\boldsymbol{0})=\sum_{\substack{\bsigma_2'\in (\BZ/q_2L\BZ)^4\\ H_{\blambda,L}(\bsigma_2')\equiv 0\bmod L}}\sum_{\substack{a_2\bmod q_2\\(a_2,q_2)=1}}\e_{q_2L^2}\left(a_2F(L\bsigma_2'+\blambda)\right)=:\CU_{L,\blambda}(q_2),$$ which is independent of $q_1$. 
	Applying (\ref{eq:S2bd}), we thus have  the following ``double Kloosterman refinement'' result: 
	\begin{equation}\label{eq:Klostermann=3}
	\begin{split}
		\sum_{q\leqslant X} \frac{S_{q,L,\blambda}(\boldsymbol{0})}{q^3}&=\sum_{\substack{q_2\leqslant X\\ q_2\mid (2L\Delta_F)^\infty}}\frac{\CU_{L,\blambda}(q_2)}{q_2^3}\sum_{\substack{q_1\leqslant \frac{X}{q_2}\\ (q_1,2L\Delta_F)=1}}\frac{S_{q_1}(\boldsymbol{0})}{q_1^3}\\ &\leqslant \sum_{\substack{q_2\leqslant X\\ q_2\mid (2L\Delta_F)^\infty}}\left|\frac{\CU_{L,\blambda}(q_2)}{q_2^3}\right|\left|\sum_{\substack{q_1\leqslant \frac{X}{q_2}\\ (q_1,2L\Delta_F)=1}}\frac{S_{q_1}(\boldsymbol{0})}{q_1^3}\right|\\ &\ll_\varepsilon X^\varepsilon L^{\frac{7}{2}+\varepsilon}\sum_{\substack{q_2\leqslant X\\ q_2\mid (2L\Delta_F)^\infty}}\gcd(q_2,L)^\frac{1}{2}\ll_\varepsilon X^\varepsilon L^{\frac{7}{2}+\varepsilon}\min(X,L)^\frac{1}{2}.
	\end{split}	\end{equation} Consequently by partial summation we have
	\begin{equation}\label{eq:Klostermann=3L}
		\sum_{q> X} \frac{S_{q,L,\blambda}(\boldsymbol{0})}{q^4}\ll_{\varepsilon} L^{4+\varepsilon}X^{-1+\varepsilon}.
	\end{equation}
	
 Proceeding as in \cite[Proof of Theorem 6, p. 201]{H-Bdelta}, for a dyadic parameter $T\ll Q$, we give two different estimates for $$\CG(B,L,R;T):=\sum_{q\sim T}\frac{S_{q,L,\blambda}(\boldsymbol{0})I_{q,L,\blambda}(w_{B,R};\boldsymbol{0})}{q^4},$$
 where we write $\sum_{q\sim T}$ for a dyadic sum $\sum_{T<q\leq 2T}$.
	According to \cite[Theorem 6.2, Proposition 6.5]{PART1}, 
	\begin{equation}\label{eq:Iq0}
		I_{q,L,\blambda}(w_{B,R};\boldsymbol{0})=\frac{B^4}{L^4R^2}\left(\CI(w_R)+O_{N}\left(\left(\frac{q}{Q}\right)^N\right)\right),
	\end{equation} where $\CI(w_R)$ is given by \eqref{eq:weightedsingint} and \eqref{eq:weightedsingintpm}. It follows that
	$$\CG(B,L,R;T)=\frac{B^4}{L^4R^2}\left(\left(\sum_{q\sim T}\frac{S_{q,L,\blambda}(\boldsymbol{0})}{q^4}\right)\CI(w_R)+O_N\left(\sum_{q\sim T}\frac{\left|S_{q,L,\blambda}(\boldsymbol{0})\right|}{q^4}\left(\frac{q}{Q}\right)^N\right)\right).$$
	Recall that \cite[Theorem 4.1]{PART1} provides
	$$\sum_{q\leqslant X}\left|S_{q,L,\blambda}(\boldsymbol{0})\right|\ll_\varepsilon L^{3+\varepsilon}X^{4+\varepsilon}.$$
	On applying \eqref{eq:Klostermann=3L} and using \begin{equation}\label{eq:CIwbd}
		\CI(w_R)\ll R^{-2}
	\end{equation} (cf. \cite[Proposition 6.5]{PART1}), we then have, 
	\begin{align*}
		&\sum_{q\leqslant QB^{-\varepsilon}}\frac{S_{q,L,\blambda}(\boldsymbol{0})I_{q,L,\blambda}(w_B;\boldsymbol{0})}{(qL)^4}\\ =&\sum_{T\nearrow QB^{-\varepsilon}}\frac{\CG(B,L,R;T)}{L^4}\\
		= &\frac{B^4}{L^8R^2}\left(\left(\sum_{q\leqslant QB^{-\varepsilon}}\frac{S_{q,L,\blambda}(\boldsymbol{0})}{q^4}\right)\CI(w_R)+O_{N,\varepsilon}\left(L^{3+\varepsilon}B^{-\varepsilon N}\right)\right)\\ =&\frac{B^4}{L^8R^2}\left(\left(\sum_{q=1}^{\infty}\frac{S_{q,L,\blambda}(\boldsymbol{0})}{q^4}\right)\CI(w_R)+O_{N,\varepsilon} \left(B^\varepsilon L^{4+\varepsilon}R^{-2} Q^{-1}+L^{3+\varepsilon}B^{-\varepsilon N}\right)\right).
	\end{align*}
	For the second alternative, by partial summation and on combining \eqref{eq:Klostermann=3} and Lemma \ref{le:Ic=0},
	\begin{align*}
		&\CG(B,L,R;T)\\ \ll &\left(\frac{B}{LR}\right)^4T\left(\sup_{T<t\leq 2T}\left|\sum_{T < q  \leq t}\frac{S_{q,L,\blambda}(\boldsymbol{0})}{q^3}\right|\right)\left(\sup_{T < q\leq 2T}\left|\frac{\widehat{\CI}_{q}(\widehat{w};\boldsymbol{0})}{q^2}\right|+\sup_{T < q\leq 2T}\left|\frac{\partial\widehat{\CI}_{q}(\widehat{w};\boldsymbol{0})}{q\partial q}\right|\right)\\ \ll_\varepsilon & \left(\frac{B}{LR}\right)^4 T^{-1+\varepsilon}L^{\frac{7}{2}+\varepsilon}\min(T,L)^\frac{1}{2}\ll \left(\frac{B}{LR}\right)^4 T^{-1+\varepsilon}L^{4+\varepsilon}.
	\end{align*} 
	Hence  \begin{align*}
		\sum_{q>QB^{-\varepsilon}}\frac{S_{q,L,\blambda}(\boldsymbol{0})I_{q,L,\blambda}(w_{B,R};\boldsymbol{0})}{(qL)^4} \ll_{\varepsilon} \left(\frac{B}{LR}\right)^4 B^\varepsilon Q^{-1}.
	\end{align*}
	We take $N$ large enough, so that the contribution of all error terms above is $$\ll_\varepsilon \left(\frac{B}{LR}\right)^4 B^{\varepsilon}Q^{-1} .$$
	
	It remains to analyse the leading constant $\sum_{q=1}^{\infty}\frac{S_{q,L,\blambda}(\boldsymbol{0})}{q^4}$.
	We factorise $$\zeta_{L,\blambda}(s):=\sum_{q=1}^{\infty}\frac{S_{q,L,\blambda}(\boldsymbol{0})}{q^s}$$ into an eulerian product. For each $q$ with $\gcd(q,2L\Delta_F)=1$ we have seen $S_{q,L,\blambda}(\boldsymbol{0})=S_q(\boldsymbol{0})$. Hence $\zeta_{L,\blambda}(s)$ has the same analytic continuation property as the function $\zeta(s,\boldsymbol{0})$ of \cite[Lemma 29]{H-Bdelta}, and an inspection of \cite[p. 195--p. 196]{H-Bdelta} reveals that $$\zeta_{L,\blambda}(s)=\BL(s-3,\psi_F)\nu_{L,\blambda}(s),$$ for some function $\nu_{L,\blambda}(s)$ that is holomorphic in the region $\Re(s)>\frac{7}{2}$. Hence  $$\sum_{q=1}^{\infty}\frac{S_{q,L,\blambda}(\boldsymbol{0})}{q^4}=\BL(1,\psi_F)\nu_{L,\blambda}(4)=\BL(1,\psi_F)\prod_{p}\left(1-\frac{\psi_F(p)}{p}\right)\CS_p(\CW;L,\blambda),$$ where $\CS_p(\CW;L,\blambda)$ is defined as $$\lim_{k\to\infty}\frac{\#\{\bv\bmod p^{k+2\operatorname{ord}_p(L)}:\bv\equiv \blambda\bmod p^{\operatorname{ord}_p(L)},F(\bv)\equiv 0\bmod p^{k+2\operatorname{ord}_p(L)}\}}{p^{3k}}.$$
	For $p\mid L$, the proof of \cite[Theorem 6.4]{PART1} shows that $$\CS_p(\CW;L,\blambda)=p^{6\operatorname{ord}_p(L)}\sigma_p(\CW;L,\bGamma),$$ and so we have proven that
	$$\sum_{q=1}^{\infty}\frac{S_{q,L,\blambda}(\boldsymbol{0})}{q^4}=L^6\widetilde{\mathfrak{S}}_{L,\bGamma}(\CW).$$ The proof is thus completed.
\end{proof}
\begin{remark}
    We can equally express the local density $\widetilde{\mathfrak{S}}_{L,\bGamma}(\CW)$ \eqref{eq:singser} in the classical way (i.e. without convergence factors) as in e.g. \cite[(2.2) and the paragraph above (3.4)]{BrowningM}. The advantage of the expression \eqref{eq:singser} is that it is absolutely convergent while the one used in  \cite{BrowningM} is only conditionally convergent. 
\end{remark}

\subsection{The Tamagawa number $\mathfrak{S}_{L,\bLambda}(\CV)$}
\begin{proposition}\label{prop:TamagawaLLambda}
	We have
	$$\mathfrak{S}_{L,\bLambda}(\CV)=\frac{\phi(L)}{\BL(2,\chi_0[L])}\widetilde{\mathfrak{S}}_{L,\bGamma}(\CW).$$ Moreover, $$L^{-2-\varepsilon}\ll_\varepsilon\mathfrak{S}_{L,\bLambda}(\CV)\ll L^{-2}.$$
\end{proposition}
\begin{proof}
	Recalling \eqref{eq:singser}, we compute
\begin{multline*}
	\frac{\phi(L)}{\BL(2,\chi_0[L])}\widetilde{\mathfrak{S}}_{L,\bGamma}(\CW) =\BL(1,\psi_F)\prod_{p\mid L}p^{\operatorname{ord}_p(L)}(1-p^{-1})(1-\psi_F(p)p^{-1})\sigma_p(\CW;L,\bGamma) \\ \times \prod_{p\nmid L}(1-p^{-2})(1-\psi_F(p)p^{-1})\sigma_p(\CW).
\end{multline*} 
Here we write $\sigma_p(\CW)= \sigma_p(\CW;L,\bGamma)$ if $p\nmid L$. Arguing as in \cite[Proof of Proposition 6.5]{PART1}, we may assume that $L_0\mid L$ for a certain fixed $L_0\in\BN$ such that $\CV\mod L$ is regular. Then we have
$$p\mid L\Rightarrow\sigma_{p}(\CV;L,\boldsymbol{\Lambda})=p^{\operatorname{ord}_p(L)}\sigma_p(\CW;L,\bGamma),$$
\begin{equation}\label{eq:sigmaVsigmaW}
	p\nmid L\Rightarrow (1-p^{-1})\sigma_p(V)=(1-p^{-2})\sigma_p(\CW).
\end{equation}
We define $\sigma_{p}(\CV;L,\boldsymbol{\Lambda})$ and $\sigma_p(V)$ as in \cite[Proposition 6.5]{PART1}. So the above factor equals
$$\BL(1,\psi_F)\prod_{p\mid L}(1-p^{-1})(1-\psi_F(p)p^{-1})\sigma_p(\CV;L,\boldsymbol{\Lambda})\prod_{p\nmid L}(1-p^{-1})(1-\psi_F(p)p^{-1})\sigma_p(V).$$
Admitting the equality in the Proposition, the upper and lower bounds for $\mathfrak{S}_{L,\bLambda}(\CV)$ are now clear (again comparig with \cite[Proposition 6.5]{PART1}).

Fix $S$ a finite set of places containing $\BR$ and the primes dividing $2\Delta_F$. 
For $p\not\in S$, let 	$$\BL_p(s,\operatorname{Pic}(V_{\overline{\BQ}})):=\frac{1}{\det(1-p^{-s}\operatorname{Fr}_p|\Pic(V_{\overline{\BQ}})_\BQ)}$$ be the $p$-adic Artin L-function. Here the Frobenius $\operatorname{Fr}_p$ acts on $\Pic(V_{\overline{\BQ}})\simeq \Pic(\CV_{\overline{\BF_p}})$. Let us define the convergence factors
$$\lambda_p:=\begin{cases}
	1 & \text{ if } p\in S;\\ \BL_p(1,\Pic(V_{\overline{\BQ}})) &\text{ if }p\not\in S.
\end{cases}$$ The (global) Artin $\BL$-function is $$\BL_S(s,\Pic(V_{\overline{\BQ}})):=\prod_{p\not\in S}\BL_p(s,\Pic(V_{\overline{\BQ}})).$$ It admits a pole at $s=1$ of order $\operatorname{rank}\Pic(V)$ (\cite[Lemme 2.2.5]{Peyre}).
Recall that (\cite[Définition 2.2]{Peyre}) the finite Tamagawa measure of $\CV(\widehat{\BZ})$ with congruence condition $(L,\bLambda)$ is $$\mathfrak{S}_{L,\bLambda}(\CV)=\left(\lim_{s\to 1}(s-1)\BL_S(s,\Pic(V_{\overline{\BQ}}))\right)\prod_{p<\infty}\lambda_p^{-1}\sigma_{p}(\CV;L,\boldsymbol{\Lambda}).$$

Let us now compute these Artin $\BL$-functions more closely. If any non-ramified $p$ is split in $\BQ(\sqrt{\Delta_F})$, which is equivalent to $\psi_F(p)=1$, then the action of $\operatorname{Fr}_p$ is trivial and hence $$\BL_p(s,\operatorname{Pic}(V_{\overline{\BQ}}))=\left(1-\frac{1}{p^{s}}\right)^{-2}.$$ While for inert $p$ ($\Leftrightarrow\psi_F(p)=-1$), $\operatorname{Fr}_p$ acts via exchanging the two generators of $\Pic(\CV_{\overline{\BF_p}})$, and so $$\BL_p(s,\operatorname{Pic}(V_{\overline{\BQ}}))=\left(1-\frac{1}{p^{2s}}\right)^{-1}.$$ 
We conclude that if $p$ is unramified, then we have 
$$\BL_p(s,\operatorname{Pic}(V_{\overline{\BQ}}))=\left(1-\frac{\psi_F(p)}{p^s}\right)^{-1}\left(1-\frac{1}{p^s}\right)^{-1}.$$
Now we easily deduce the accordance of the non-archimedean arithmetic factor with $\mathfrak{S}_{L,\bLambda}(\CV)$. 
Finally we note that (for example by direct computation) $H^1(\BQ,\Pic(V_{\overline{\BQ}}))=0$
, and hence $\beta(V)=1$. So we finally confirm Peyre's prediction.

\end{proof}

\section{Counting points with real conditions}\label{se:realzoom}
In this section we always assume that no congruence condition is imposed.
We thus take $L=1,\blambda=\boldsymbol{0}$ and omit them in what follows. Hence we denote the counting function \eqref{eq:countingW} (applied to $w_{B,R}$) by $\CN_{\CW}(w_{B,R})$. 
\begin{theorem}\label{thm:realzoom}
Assume that $R^2=O(B^{1-\tau})$ for some $-1<\tau<1$. Then
		$$\CN_{\CW}(w_{B,R})=B^{2}\CI(w_R)\widetilde{\mathfrak{S}}(\CW)+O_\varepsilon\left(\frac{B^{2-\tau+\varepsilon}}{R^2}\right),$$ where $\CI(w_R)$ and $\widetilde{\mathfrak{S}}(\CW)$ are defined respectively by \eqref{eq:weightedsingint} and \eqref{eq:singser}.

\end{theorem}

\begin{remarks}
	\begin{enumerate}
		\item In a similar way to \cite[Theorem 2.3]{PART1}, Theorem \ref{thm:realzoom} provides an asymptotic formula when $0<\tau<1$, while it is an upper bound when $-1<\tau\leqslant 0$.
		\item If one regards the real condition as fixed, i.e., we take $\tau=1$ so that $R=O(1)$, then Theorem \ref{thm:realzoom} also improves upon the error term  $O(B^{\frac{3}{2}})$ of \cite[Theorem 6]{H-Bdelta} (see also \cite{Lindqvist} the comments after Corollary 1.4).
	\end{enumerate}
\end{remarks}
\subsection{More on $S_q(\bc)$}
In this case the quadratic exponential sums $S_{q,L,\blambda}(\bc)$ \eqref{eq:Sqc} simplify to \begin{equation}\label{eq:sqc}
	S_q(\bc):= \sum_{\bb \in (\BZ/q\BZ)^{4}}\sum_{\substack{a\bmod q\\(a,q)=1}} e\left( \frac{aF(\bb) + \bc\cdot \bb}{q}\right).
\end{equation}
We show that in this case the Poisson variables $\bc$ satisfying $F^*(\bc)=0$ have a negligible contribution.
\begin{lemma}\label{prop:Kloostermann=3}
	Uniformly for any $\bc\in\BZ^4$ with $F^*(\bc)=0$, we have $$\sum_{q\leqslant X}\frac{S_q(\bc)}{q^3}\ll_\varepsilon X^{\varepsilon}.$$
\end{lemma}
\begin{remark}
	In our setting \cite[Lemma 30]{H-Bdelta} provides $\sum_{q\leqslant X} S_q(\bc)\ll_\varepsilon X^{\frac{7}{2}+\varepsilon}(1+|\bc|)$ whenever $F^*(\bc)=0$. So Lemma \ref{prop:Kloostermann=3} improves on the estimate that one would obtain from \cite[Lemma 30]{H-Bdelta} in saving a factor of $|\bc|$ and an $X^{\frac{1}{2}}$-factor.
\end{remark}
\begin{proof}
	Our proof is a direct analogue of the treatment of \eqref{eq:Klostermann=3}.
	According to \cite[Lemma 23 and p. 193--p. 194]{H-Bdelta}, $S_q(\bc)$ is multiplicative for any $\bc$:
	\begin{equation}\label{eq:multiplicative}
		S_{q_1q_2}(\bc)=S_{q_1}(\bc)S_{q_2}(\bc),\quad \text{whenever } \gcd(q_1,q_2)=1.
	\end{equation}
	
	If $F^*(\bc)=0$, then by \eqref{eq:S1R}, \eqref{eq:Rqc}, and \eqref{eq:S2bd}, we have $$S_{q}(\bc)\begin{cases}
		 =q^3\theta_1(q)\psi_F(q)&\text{ if } \gcd(q,2\Delta_F)=1;\\  \ll q^3 &\text{ if } q\mid (2\Delta_F)^\infty.
	\end{cases}$$  
Then by Lemma \ref{prop:fstarc=0} (with $\chi$ being principal),
$$\sum_{\substack{q\leqslant Y\\ \gcd(q,2\Delta_F)=1}}\frac{S_q(\bc)}{q^3}=\sum_{\substack{n\leqslant Y\\ \gcd(n,2\Delta_F)=1}}\psi_F(n)\theta_1(n)\ll_\varepsilon Y^\varepsilon.$$
We thus have, thanks to the multiplicativity \eqref{eq:multiplicative}, 
\begin{align*}
	\sum_{q\leqslant X}\frac{S_q(\bc)}{q^3}&=\sum_{\substack{q_2\leqslant X\\ q_2\mid (2\Delta_F)^\infty}}\frac{S_{q_2}(\bc)}{q_2^3}\sum_{\substack{q_1\leqslant \frac{X}{q_2}\\ (q_1,2\Delta_F)=1}}\frac{S_{q_1}(\bc)}{q_1^3}\\ &\leqslant \sum_{\substack{q_2\leqslant X\\ q_2\mid (2\Delta_F)^\infty}}\left|\frac{S_{q_2}(\bc)}{q_2^3}\right|\left|\sum_{\substack{q_1\leqslant \frac{X}{q_2}\\ (q_1,2\Delta_F)=1}}\frac{S_{q_1}(\bc)}{q_1^3}\right|\\ &\ll_\varepsilon X^\varepsilon\sum_{\substack{q_2\leqslant X\\ q_2\mid (2\Delta_F)^\infty}}1\ll_\varepsilon X^\varepsilon.
\end{align*}
The proof is thus completed.
\end{proof}
\subsection{Proof of Theorem \ref{thm:realzoom}}
Our assumption $-1<\tau<1$ guarantees that $Q\gg B^{\frac{1+\tau}{2}}\gg 1$. In view of \eqref{eq:poisson} and \eqref{eq:IJ}, the contribution from all $\bc\neq\boldsymbol{0}$ is
\begin{equation}\label{eq:Rcneq0}
	\frac{C_Q}{Q^2}\sum_{q=1}^{\infty}\sum_{\bc\in\BZ^4\setminus\boldsymbol{0}}\frac{S_{q}(\bc)I_{q}(w_{B,R};\bc)}{q^{4}}\ll \frac{B^2}{R^2}\left|\sum_{q=1}^{\infty}\sum_{\bc\in\BZ^4\setminus\boldsymbol{0}}\frac{S_{q}(\bc)\widehat{\CI}_{q}(\widehat{w};\boldsymbol{d})}{q^4}\right|.
\end{equation}
When $|\boldsymbol{d}| \gg QB^{\varepsilon}$ for some $\varepsilon >0$, Heath-Brown's ``simple estimate'' (see Lemma \ref{le:easy}) can be applied (with $N$ sufficiently large in terms of $\varepsilon$) to obtain an arbitrary power saving. Recalling \eqref{eq:d} and arguing as in \cite[Section 5, Step I]{PART1}, it remains to consider $\bc$ satisfying \begin{equation}\label{eq:cond0}
	0\neq|\bc|\leqslant B^{\frac{1}{2}(1-\tau)+\varepsilon},\quad |\widetilde{c_0}|\leqslant \frac{B^{\varepsilon}}{R}\quad \text{and}\quad |\widetilde{c_j}|\leqslant B^\varepsilon \quad \text{for all }j=2,3.
\end{equation} 
 
 For a dyadic parameter $1\ll T\ll Q$, we claim that in both cases $F^*(\bc)\neq 0$ and $F^*(\bc)=0$, we have
\begin{equation}\label{eq:fstarc=0}
	\sum_{q\sim T}\frac{S_{q}(\bc)\widehat{\CI}_{q}(\widehat{w};\boldsymbol{d})}{q^4}\ll_\varepsilon\frac{B^\varepsilon T^\varepsilon}{|\bd|^{1-\varepsilon}}.
\end{equation} Here we recall the vector $\bd=\bd(B,R;\bc)$ from \eqref{eq:d} (with $L=1$).  Indeed, it follows directly from partial summation that the quantity to be estimated is
\begin{align*}
	\ll &T\left(\sup_{t\sim T}\left|\sum_{T\leqslant q<t}\frac{S_q(\bc)}{q^3}\right|\right)\left(\sup_{q\sim T}\left|\frac{\widehat{\CI}_{q}(\widehat{w};\boldsymbol{d})}{q^2}\right|+\sup_{q\sim T}\left|\frac{\partial\widehat{\CI}_{q}(\widehat{w};\boldsymbol{d})}{q\partial q}\right|\right).
\end{align*}
 Applying Proposition \ref{prop:fstarcneq0} when $F^*(\bc)\neq 0$ and Lemma \ref{prop:Kloostermann=3} when $F^*(\bc)=0$ yields that the quantity in the first bracket is $\ll_{\varepsilon} (1+|\bc|^\varepsilon) T^\varepsilon\ll B^\varepsilon T^\varepsilon$. Applying Lemma \ref{le:hard} to $\widehat{\CI}_{q}(\widehat{w};\boldsymbol{d})$ and its derivative yields that the quantity in the second bracket is $\ll_\varepsilon \frac{Q^\varepsilon}{T|\bd|^{1-\varepsilon}}$. We thus get the desired bound \eqref{eq:fstarc=0}.

We use the simple observation \begin{equation}\label{eq:dc}
	|\bd|\gg B^{\tau}|\bc|
\end{equation} throughout. 
It follows from \eqref{eq:fstarc=0} that  \begin{equation}\label{eq:sumcnot0}
	\sum_{\substack{\bc:\eqref{eq:cond0} \text{ holds}}}\sum_{q\ll Q}\frac{S_{q}(\bc)\widehat{\CI}_{q}(\widehat{w};\boldsymbol{d})}{q^{4}}\ll_\varepsilon B^{-\tau+\varepsilon}\sum_{\bc:\eqref{eq:cond0}\text{ holds}}|\bc|^{-1}\ll_{\varepsilon} B^{-\tau+\varepsilon}.
\end{equation}
Therefore, on inserting \eqref{eq:sumcnot0} back to \eqref{eq:Rcneq0}, the contribution from $\bc\neq\boldsymbol{0}$ is  $\ll_\varepsilon \frac{B^{2-\tau+\varepsilon}}{R^2}$ and goes to the remainder term. 

Finally, the condition $R^2\ll B^{1-\tau}$ implies $Q\gg B^{\frac{1+\tau}{2}}$,  Proposition \ref{prop:c=0L} shows that $\bc=\boldsymbol{0}$ contributes \begin{align*}
	\frac{C_Q}{Q^2}\sum_{q=1}^\infty\frac{S_{q}(\boldsymbol{0})I_{q}(w_{B,R};\boldsymbol{0})}{q^4}=B^{2}\left(\CI(w_R)\widetilde{\mathfrak{S}}(\CW)+O_\varepsilon\left(\frac{B^{-\frac{1}{2}(1+\tau)+\varepsilon}}{R^2}\right)\right).
\end{align*} Combining everything finishes the proof. \qed
\begin{remark}
We could have compared $\sum_{q\leqslant X}\frac{S_q(\boldsymbol{0})}{q^4}$ with $\widetilde{\mathfrak{S}}(\CW)$ by means of Perron's formula as in \cite[p. 198]{H-Bdelta}, given that the analytic properties of the zeta function $\zeta(s,\bc)$ (cf. \cite[Lemma 29]{H-Bdelta}) are well-understood. Indeed, $\widetilde{\mathfrak{S}}(\CW)$ equals the residue of the complex function $\varPi(s+4)X^s/s$ at $s=0$. However the error term deduced there is insufficient for optimality.
\end{remark}
\subsection{Proof of Theorem \ref{thm:realzoommain}}
We have by Möbius inversion \begin{align*}
	\CN_{V}(w_{B,R})&=\frac{1}{2}\sum_{d\ll B}\mu(d)\sum_{\substack{\bx\in\BZ^4\setminus\boldsymbol{0}\\ F(\bx)=0,d\mid \bx}}w_{B,R}(\bx)\\ &=\frac{1}{2}\sum_{d\ll B}\mu(d)\left(\CN_{\CW}(w_{B/d,R})+O(1)\right).
\end{align*}
For any $d\ll B$, let $\tau_d\in \BR$ be such that $$R^2=\left(\frac{B}{d}\right)^{1-\tau_d}.$$
Note that $$\left(\frac{B}{d}\right)^{-\tau_d}\ll dB^{-\tau}.$$
For $0<\varepsilon<\tau$, let 
\begin{equation}\label{eq:D0}
	D_0:=B^{\frac{1}{2}(1+\tau-\varepsilon)}.
\end{equation}
 For $d\leqslant D_0$, we have $1>\tau_d\geqslant-1+\frac{2\varepsilon}{1-\tau+\varepsilon}>-1$. The contribution from $d\leqslant D_0$ is, by Theorem \ref{thm:realzoom} and \eqref{eq:CIwbd}
\begin{align*}
	=&\frac{1}{2}B^{2}\CI(w_R)\widetilde{\mathfrak{S}}(\CW)\left(\sum_{d\leqslant D_0}\frac{\mu(d)}{d^2}\right)+O_\varepsilon\left(\sum_{d\leq D_0}\frac{B^{2-\tau_d+\varepsilon}}{d^{2-\tau_d}R^2}\right)+O(D_0)\\ =&\frac{1}{2}B^{2}\CI(w_R)\frac{\widetilde{\mathfrak{S}}(\CW)}{\zeta(2)}+O_\varepsilon\left(\frac{B^2}{R^2}\left(\sum_{d>D_0}\frac{1}{d^2}+B^{-\tau+\varepsilon}\sum_{d\leq D_0}\frac{1}{d}\right)\right)+O(D_0)\\ =&\frac{1}{2}B^{2}\CI(w_R)\frac{\widetilde{\mathfrak{S}}(\CW)}{\zeta(2)}+O_\varepsilon\left(\frac{B^{2-\tau+\varepsilon}}{R^2}\right).
\end{align*}
Finally, for large $d$, we use an alternative bound for $\CN_{\CW}(w_{B/d,R})$ based on lattice point counting. The Witt form \eqref{eq:FWitt}, and our choice of the weight function \eqref{eq:weightBR} force that \begin{equation}\label{eq:bdt0t2t3}
	t_0(\bx)\ll \frac{B}{d},\quad t_2(\bx),t_3(\bx)\ll  1+\frac{B}{dR}.
\end{equation} Note that $t_0\neq 0$.  Then once $t_0,t_2,t_3$ are chosen, $t_1$ is uniquely determined and satisfies
$$t_1(\bx)\ll \left|\frac{F_2(t_2,t_3)}{t_0}\right| \ll \left|t_0F_2\left(\frac{t_2}{t_0},\frac{t_3}{t_0}\right)\right|\ll \frac{B}{dR^2}.$$
Hence the contribution from $d>D_0$ is bounded by
\begin{align*}
&\sum_{d>D_0}\left(\frac{B}{d}\right) \left(1+\frac{B}{dR}\right)^2\left(1+\frac{B}{dR^2}\right)\\
&\ll B\sum_{d>D_0}\frac{1}{d} + \sum_{d>D_0}\frac{B^2}{d^2R^2} + \sum_{d>D_0}\frac{B^3}{d^3R^2}+\sum_{d>D_0}\frac{B^4}{d^4R^4}\\
&\ll B^{1+\varepsilon}+\frac{B^2}{D_0 R^2} + \frac{B^3}{D_0^2R^2} + \frac{B^4}{D_0^3R^4}\\
&\ll B^{1+\varepsilon}.
\end{align*}
To finish the proof it suffices to recall from Proposition \ref{prop:TamagawaLLambda} that
 $$\mathfrak{S}(\CV)=\frac{\widetilde{\mathfrak{S}}(\CW)}{\zeta(2)}.$$
 
 The deduction of the upper bound for $\aess^{\BR}$ is the same as \cite[Theorem A.3]{PART1}. \qed 

\section{Counting points with non-archimedean conditions}\label{se:zoompadic}
We assume from now on that no real condition is imposed, i.e. $R=1$.
\subsection{Counting integral points on the affine cone} 
Our goal in this section is to establish the following result.
\begin{theorem}\label{thm:Llambda}
Let $\wtil_B$ be defined by \eqref{eq:weightpm}. Assume $L^2=O(B^{1-\tau})$ for a certain $-1<\tau<1$. Then we have
	$$\CN_{\CW}(\wtil_B;(L,\bGamma))=B^2\left(\widetilde{\mathfrak{S}}_{L,\bGamma}(\CW)\CI(w_0)+\CK_{L,\bGamma}(\widehat{w})+O_{\varepsilon}\left(B^{-\tau+\varepsilon}L^{-1}  \right)\right),$$ where  
	$\CI(w_0)$ is given by \eqref{eq:weightedsingintpm} (with $R=1$) and we refer to \eqref{eq:CKc} for the definition of the term $\CK_{L,\bGamma}(\widehat{w})$.
\end{theorem}
\subsubsection{Preliminary manipulations for the $S^{(2)}$-term}

If $L>1,\bc\neq\boldsymbol{0}$, in general $S_{q,L,\blambda}(\bc)$ \eqref{eq:Sqc} and especially $S^{(2)}_{q,L,\blambda}(\bc)$ \eqref{eq:S2} is not multiplicative as a function of $q$. 
Our strategy, inspired by \cite{Lindqvist}, is to make use of orthogonality of characters to ``drop off'' the dependence on $q_1$, upon twisting $S_{q,L,\blambda}(\bc)$ by a family of Dirichlet characters and adding an extra average over such characters. 

To carry out the details, observe that
$$T^{(2)}_{q,L,\blambda}(a_2,\bsigma_2,\bc)=\e_{q_2L^2}\left(a_2F(Lq_1\bsigma_2+\blambda)+\bc\cdot L\bsigma_2\right).$$
Then the change of variables $$a_2':=a_2q_1^2\bmod q_2,\quad \balpha:=L\bsigma_2+\overline{q_1}\blambda\bmod q_2L^2$$ gives 
\begin{equation}\label{eq:S2S}
	\e_{q_2L^2}\left(\overline{q_1}\bc\cdot\blambda\right)S^{(2)}_{q,L,\blambda}(\bc)= \CS_{q_2,L,\blambda}(q_1;\bc)
\end{equation} where, on recalling \eqref{eq:equivdiv} and the summation condition on $\bsigma_2$, we define, for every invertible $x\bmod L$, $$\CS_{q_2,L,\blambda}(x;\bc):=\sum_{\substack{a_2'\bmod q_2\\(a_2',q_2)=1}}\sum_{\substack{\balpha\bmod q_2L^2\\F(\balpha)\equiv 0\bmod L^2\\ \balpha\equiv \overline{x}\blambda\bmod L}}\e_{q_2L^2}\left(a_2' F(\balpha)+\bc\cdot\balpha\right).$$
We note that, with fixed $q_2$ and $\bc$, as a function of $q_1$, $\CS_{q_2,L,\blambda}(q_1;\bc)$ only depends on its class modulo $L$.
For each fixed $q_2$, let us define, for every $\bc\in\BZ^4$ and for every Dirichlet character $\chi\bmod L$, \begin{equation}\label{eq:CAchic}
	\CA_{q_2,L,\blambda}(\chi;\bc):=\frac{1}{\phi(L)}\sum_{x\bmod L}\chi(x)^c\CS_{q_2,L,\blambda}(x;\bc).
\end{equation} 
The advantage of introducing \eqref{eq:CAchic} is that now we have
$$\CS_{q_2,L,\blambda}(q_1;\bc)=\sum_{\chi\bmod L}\chi(q_1)\CA_{q_2,L,\blambda}(\chi;\bc).$$
Then gathering together \eqref{eq:S1R} and \eqref{eq:S2S}, we obtain, for every $q,\bc$,
\begin{equation}\label{eq:S}
	\begin{split}
		&\e_{qL^2}(\bc\cdot\blambda)S_{q,L,\blambda}(\bc)\\ =&\e_{qL^2}(\bc\cdot\blambda)\e_{q_1}\left(-\overline{q_2L^2}\bc\cdot\blambda\right)\CR(q_1;\bc)S^{(2)}_{q,L,\blambda}(\bc)\\ =&\CR(q_1;\bc)\CS_{q_2,L,\blambda}(q_1;\bc) \\ =&\sum_{\chi\bmod L}\CA_{q_2,L,\blambda}(\chi;\bc)\chi(q_1)\CR(q_1;\bc),
	\end{split}
\end{equation}
where we recall \eqref{eq:Rqc}. Here we have used the classical ``inversion formula'' $$\frac{\overline{q_2L^2}}{q_1}+\frac{\overline{q_1}}{q_2L^2}\equiv \frac{1}{q_1q_2L^2} \bmod 1,
$$ 
where we write $\bar{q_2L^2}$ for the inverse of $q_2L^2$ modulo $q_1$ and $\bar{q_1}$ for the inverse of $q_1$ modulo $q_2L^2$. Therefore it becomes possible to average over $q_1$ while fixing $\chi$.

We end this subsection with the following ``flipping'' formula, which is a direct analogue of \cite[Lemma 3.5]{Lindqvist}.
\begin{lemma}\label{le:flip}
	Let $q_2,\bc$ be fixed and let $d\in\BN$ be such that $\gcd(d,L)=1$. Then for every $\chi\bmod L$ we have
	$$	\CA_{q_2,L,d\blambda}(\chi;\bc)=\chi(d)^c \CA_{q_2,L,\blambda}(\chi;\bc).$$
\end{lemma}
\begin{proof}
	On observing that for every invertible $x\bmod L$, $\overline{x}d=\overline{x\overline{d}}\bmod L$, we have $\CS_{q_2,L,d\blambda}(x;\bc)=\CS_{q_2,L,\blambda}(x\overline{d};\bc)$ and hence the change of variables $x\mapsto x':=x\overline{d}$ gives
	\begin{align*}
			\CA_{q_2,L,d\blambda}(\chi;\bc)&=\frac{1}{\phi(L)}\sum_{x\bmod L}\chi(x)^c\CS_{q_2,L,\blambda}(x\overline{d};\bc)\\ &=\frac{1}{\phi(L)}\sum_{x'\bmod L}\chi(dx')^c\CS_{q_2,L,\blambda}(x';\bc)\\
            &=\chi(d)^c \CA_{q_2,L,\blambda}(\chi;\bc). \qedhere
	\end{align*}
\end{proof}
\subsubsection{The contribution from $\bc\neq\boldsymbol{0}$}

We recall \eqref{eq:CIq} and $w_0$. Let \begin{equation}\label{eq:CJc}
	\CJ_{L}(\widehat{w};\bc):=\int_{0}^{\infty}\frac{\widehat{\CI}_{t}(\widehat{w};\boldsymbol{d})}{t}\operatorname{d}t=\int_{0}^{\infty}\int_{\BR^4}\widehat{w}(\bt)h\left(r,F(\bt)\right)\e_r\left(-\frac{\bc}{L}\cdot\bt\right)\frac{\operatorname{d}r\operatorname{d}\bt}{r},
\end{equation} on recalling the real vector $\bd=\boldsymbol{d}(B,L;\cc)$ from \eqref{eq:d} and observing that $\frac{\bd}{Q}= \frac{\bc}{L}$.
On using Lemmas \ref{le:easy} and \ref{le:hard}, for $|\bc|>LQ^{\varepsilon}$ we can show (as in \cite[(13.3)]{H-Bdelta})
\begin{equation}\label{eq:CJL}
	\CJ_{L}(\widehat{w};\bc)\ll_N \left(\frac{|\bc|}{L}\right)^{-N}.
\end{equation} 
Moreover, for $|\bc|<LQ^{\varepsilon}$ we still have the upper bound
$$\CJ_{L}(\widehat{w};\bc)\ll \frac{LQ^{\varepsilon}}{|\boldsymbol{c}|}.$$
Let $\varepsilon_{\Delta_F}\mid 8\Delta_F$ be the conductor of the character $\psi_F$.  
We recall \eqref{eq:CAchic} and define  \begin{equation}\label{eq:CB}
	\CB_{L,\blambda}(\bc):=\begin{cases}
		\theta_2(2\Delta_F L)\sum_{\substack{u\mid (2\Delta_F L)^\infty}}\frac{\CA_{u,L,\blambda}(\chi_0[L]\widetilde{\psi_F};\bc)}{u^4} &\text{ if } \varepsilon_{\Delta_F}\mid L;\\ 0 &\text{ otherwise}.
	\end{cases}
\end{equation} 
To see that the $u$-sum is convergent, by \eqref{eq:S2bd}, \eqref{eq:S2S}, and \eqref{eq:CAchic}, we have
\begin{equation}\label{eq:CAbound}
	\left|\CA_{q_2,L,\blambda}(\chi;\bc)\right|\leqslant \max_{x\bmod L}\left|\CS_{q_2,L,\blambda}(x;\bc)\right|\ll (q_2L)^3\gcd(q_2,L)^\frac{1}{2}.
\end{equation}
Hence \begin{equation}\label{eq:CBbd}
	\sum_{u\mid (2\Delta_F L)^\infty}\frac{\left|\CA_{u,L,\blambda}(\chi_0[L]\widetilde{\psi_F};\bc)\right|}{u^4}\ll L^3 \sum_{u\mid (2\Delta_F L)^\infty}\frac{1}{u^\frac{1}{2}}=L^3\prod_{p\mid 2\Delta_F  L}\frac{1}{1-p^{-\frac{1}{2}}}\ll_\varepsilon L^{3+\varepsilon}.
\end{equation}

Consider for $\bc\in\BZ^{4}$ the lattice condition \begin{equation}\label{eq:latticecond}
	\bc\equiv b\nabla F(\blambda)\bmod L \text{ for some } b\in\BZ.
\end{equation} We let \begin{equation}\label{eq:CKc}
	\CK_{L,\bGamma}(\widehat{w}):=\frac{1}{L^6}\sum_{\substack{\bc\neq\boldsymbol{0}:F^*(\bc)=0\\\eqref{eq:latticecond}\text{ holds}}}\CB_{L,\blambda}(\bc)\CJ_{L}(\widehat{w};\bc).
\end{equation} The convergence of $\CK_{L,\bGamma}(\widehat{w})$ follows from \eqref{eq:CJL} and \eqref{eq:CBbd}, and $$\CK_{L,\bGamma}(\widehat{w})\ll_{N_0} L^{N_0}Q^{\varepsilon}$$ for any appropriately large $N_0$. 
\begin{remark}\label{rmk:deltanotdiv}
	We have avoided computing $\CK_{L,\bGamma}(\widehat{w})$ via expanding the expression of $\CB_{L,\blambda}(\bc)$ from \eqref{eq:CB} which would be quite tedious (see also \cite[Remark below Lemma 3.4]{Lindqvist}). Instead, we shall see below that the appearance of $\CK_{L,\bGamma}(\widehat{w})$ is a result of the existence of a Brauer--Manin obstruction to strong approximation on $\CW^o$. See Remark \ref{rmk:noCK} below.)
\end{remark}

The goal of this subsection is to prove:
\begin{proposition}\label{prop:cneq0qsum}
Let $-1<\tau<1$ and $L^2=O(B^{1-\tau})$. We have
	\begin{align*}
		&\underset{\substack{\bc\in\BZ^{4}\setminus\boldsymbol{0},q\ll Q}}{\sum\sum}\frac{S_{q,L,\blambda}(\bc)I_{q,L,\blambda}(\wtil_B;\bc)}{(qL)^4}=\frac{B^4}{L^2}\left(\CK_{L,\bGamma}(\widehat{w})+O_{\varepsilon}\left(B^{-\tau+\varepsilon}L^{-1+\varepsilon}\right)\right),
	\end{align*} 	where $\CK_{L,\bGamma}(\widehat{w})$ is defined by \eqref{eq:CKc}.
\end{proposition}

At a crucial point in the proof of Proposition \ref{prop:cneq0qsum}, we rely on the following estimate for certain ``sums of Kloosterman sums'':

\begin{proposition}\label{prop:qsumS}
	Let $\bc\in\BZ^4$ be such that $F^*(\bc)=0$. 
	Then
	$$\sum_{q\leqslant X}\e_{qL^2}(\bc\cdot\blambda)S_{q,L,\blambda}(\bc)=\frac{1}{4}\CB_{L,\blambda}(\bc)X^4+O_{\varepsilon}\left(X^{3+\varepsilon}L^{5 + \varepsilon}\right).$$ 
\end{proposition}

\begin{proof} According to \eqref{eq:S} and \eqref{eq:RqcFstarc=0}, we have
	\begin{align*}
		\sum_{q\leqslant X}\e_{qL^2}(\bc\cdot\blambda)S_{q,L,\blambda}(\bc)&=\sum_{\substack{q=q_1q_2\leqslant X\\ \gcd(q_1,2L\Delta_F)=1\\ q_2 \mid (2L\Delta_F)^\infty,q_2\leqslant X}}\sum_{\chi\bmod L}\CA_{q_2,L,\blambda}(\chi;\bc)\chi(q_1)\CR(q_1;\bc)\\ &=\sum_{\substack{q_2\leqslant X\\q_2 \mid (2L\Delta_F)^\infty}}\sum_{\chi\bmod L}\CA_{q_2,L,\blambda}(\chi;\bc)\sum_{\substack{q_1\leqslant X/q_2\\\gcd(q_1,2L\Delta_F)=1}}\chi(q_1)\psi_F(q_1)\phi(q_1^3).
	\end{align*}

 If  $\varepsilon_{\Delta_F}\mid L$, then for any fixed $q_2$ in the sum, the $q_1$-sum with character $\chi = \chi_0[L]\widetilde{\psi_F}$ contributes
	\begin{align*}
	&\sum_{\substack{n\leqslant X/q_2}}\chi_0[2\Delta_F L](n)\phi(n^3)\\&=\sum_{\substack{e_1\leqslant X/q_2\\\gcd(e_1,2\Delta_F L)=1}}\frac{\mu(e_1)}{e_1}e_1^3\sum_{\substack{e_2\leqslant \frac{X}{q_2e_1}\\\gcd(e_2,2\Delta_F L)=1}}e_2^3\\ &=\sum_{\substack{e_1\leqslant X/q_2\\\gcd(e_1,2\Delta_F L)=1}}\mu(e_1)e_1^2\left(\frac{1}{4}\theta_1(2\Delta_F L)\left(\frac{X}{q_2e_1}\right)^4+O\left(L^{1+\varepsilon}\left(\frac{X}{q_2e_1}\right)^3\right)\right)\\ &=\frac{X^4}{4q_2^4}\theta_1(2\Delta_F L)\left(\sum_{\substack{e_1\leqslant X/q_2\\\gcd(e_1,2\Delta_F L)=1}}\frac{\mu(e_1)}{e_1^2}\right)+O_\varepsilon\left(\left(\frac{X}{q_2}\right)^{3+\varepsilon}X^{\varepsilon}L^{1+\varepsilon}\right)\\ &=\frac{X^4}{4q_2^4}\theta_2(2\Delta_F L)+O_\varepsilon\left(\left(\frac{X}{q_2}\right)^{3+\varepsilon}X^{\varepsilon}L^{1+\varepsilon}\right).
\end{align*}
Extending the sum over all such $q_2$ to infinity gives the main term, upon introducing an error term, that we can bound using \eqref{eq:CAbound} by
\begin{align*}
	X^4\sum_{\substack{u>X\\u\mid (2\Delta_F L)^\infty}}\frac{\left|\CA_{u,L,\blambda}(\chi_0[L]\widetilde{\psi_F};\bc)\right|}{u^4}&\ll X^4L^3\sum_{\substack{u> X\\u\mid (2\Delta_F L)^\infty}}\frac{\gcd(u,L)^\frac{1}{2}}{u} \\ &\ll_{\varepsilon} X^{3+\varepsilon} L^{\frac{7}{2}+\varepsilon}.
\end{align*}Similarly the error term above contributes
$$\ll_\varepsilon X^{3+\varepsilon}L^\varepsilon\sum_{\substack{q_2\leqslant X\\q_2 \mid (2L\Delta_F)^\infty}}\frac{(q_2L)^3\gcd(q_2,L)^\frac{1}{2}L}{q_2^3}\ll_{\varepsilon} X^{3+\varepsilon}L^{\frac{9}{2}+\varepsilon}.$$
By partial summation, the remaining characters contribute, by Lemma \ref{prop:fstarc=0}, 
\begin{align*}
	&\ll_\varepsilon \sum_{\substack{q_2\leqslant X\\q_2 \mid (2L\Delta_F)^\infty}}\sum_{\substack{\chi\bmod L\\ \chi\neq \chi_0[L]\widetilde{\psi_F}}}\left|\CA_{q_2,L,\blambda}(\chi;\bc)\right| \left(X/q_2\right)^{3+\varepsilon}L^{\frac{1}{2}+\varepsilon}\\ &\ll_\varepsilon X^{3+\varepsilon}L^{\frac{9}{2}+\varepsilon}\sum_{\substack{q_2\leqslant X\\q_2 \mid (2L\Delta_F)^\infty}}\gcd(q_2,L)^\frac{1}{2}\\ &\ll_{\varepsilon} X^{3+\varepsilon}L^{5+\varepsilon}.
\end{align*}
Gathering the bounds obtained so far gives the result.
\end{proof}

\begin{proof}[Proof of Proposition \ref{prop:cneq0qsum}]
	By \eqref{eq:IJ} and \eqref{eq:d}, the summand equals (recall that $M$ is the identity matrix so $\det M=1$)
	$$\left(\frac{B}{L}\right)^4\frac{\e_{qL^2}(\bc\cdot\blambda)S_{q,L,\blambda}(\bc)\widehat{\CI}_{q}(\widehat{w};\boldsymbol{d})}{(qL)^4}.$$ For $T\ll Q$ a dyadic parameter, we define
\begin{equation}\label{eq:CW}
		\CG(B,L,T;\bc):=\sum_{q\sim T}\frac{\e_{qL^2}(\bc\cdot\blambda)S_{q,L,\blambda}(\bc)\widehat{\CI}_{q}(\widehat{w};\boldsymbol{d})}{q^4}.
\end{equation}
	
	By \cite[Proposition 5.2]{PART1}, $\CG(B,L,T;\bc)\neq 0$ only if $\bc$ satisfies the lattice condition \eqref{eq:latticecond}.
	Arguing as in \cite[Step I of the proof of Theorem 5.1]{PART1}, we can reduce the range of $\bc$ to 
	\begin{equation}\label{eq:c00}
		0\neq|\bc|\leqslant B^{\frac{1}{2}(1-\tau)+\varepsilon},\quad |\widetilde{c_j}|\leqslant L B^\varepsilon \quad \text{for }j=0,2,3.
	\end{equation} 

	Assume first $F^*(\bc)\neq 0$. To deal with $\CG(B,L,T;\bc)$, a straightforward generalisation of the proof of Theorem \ref{thm:realzoom} yields, on combining with Proposition \ref{prop:fstarcneq0}, 
	\begin{align*}
		&\CG(B,L,T;\bc)\\ \ll &T\left(\sup_{T<t\leq 2T}\sum_{T< q\leq t}\frac{\left|S_{q,L,\blambda}(\bc)\right|}{q^3}\right)\left(\sup_{T< q\leq 2T}\left|\frac{\widehat{\CI}_{q}(\widehat{w};\boldsymbol{d})}{q^2}\right|+\sup_{T <  q\leq 2T}\left|\frac{\partial\widehat{\CI}_{q}(\widehat{w};\boldsymbol{d})}{q\partial q}\right|\right)\\ \ll &\frac{B^\varepsilon L^{3+\varepsilon}T^\varepsilon }{|\bd|}\min (T,L)^\frac{1}{2}.
	\end{align*} 
Therefore, using the observation $$	|\bd|\gg B^{\tau}|\bc|,$$ the $q$-sum for every such fixed $\bc$ contributes
$$\ll_{\varepsilon} \frac{B^\varepsilon L^{3+\varepsilon}}{|\bd|}\min(Q,L)^\frac{1}{2}\ll B^{-\tau+\varepsilon}L^{3+\varepsilon}\min(Q,L)^\frac{1}{2}|\bc|^{-1}.$$
By \cite[Lemma 5.3]{PART1}, $$\sum_{\bc:\eqref{eq:latticecond},\eqref{eq:c00}\text{ hold}}|\bc|^{-1}\ll_{\varepsilon} B^\varepsilon.$$ So we conclude
\begin{align*}
	&\underset{\substack{\bc\in\BZ^{4}\setminus\boldsymbol{0},F^*(\bc)\neq 0\\ q\ll Q}}{\sum\sum}\frac{S_{q,L,\blambda}(\bc)I_{q,L,\blambda}(\wtil_B;\bc)}{(qL)^4}  \ll_\varepsilon \left(\frac{B}{L^2}\right)^4 B^{-\tau+\varepsilon}L^{3+\varepsilon}\min(Q,L)^\frac{1}{2}.
\end{align*}

Assume from now on that $F^*(\bc)=0$. We write $$A_t:=\sum_{q\leqslant t}\e_{qL^2}(\bc\cdot\blambda)S_{q,L,\blambda}(\bc).$$ Then using partial summation, we get
\begin{equation}\label{eq:CBLTc}
    	\CG(B,L,T;\bc)=-\int_{T}^{2T}\left(A_t \frac{\partial}{\partial t}\left(\frac{\widehat{\CI}_{t}(\widehat{w};\boldsymbol{d})}{t^4}\right) \right)\operatorname{d}t +\left[A_t \frac{\widehat{\CI}_{t}(\widehat{w};\boldsymbol{d})}{t^4}\right]_{T}^{2T}
\end{equation}
By Proposition \ref{prop:qsumS} and again by Lemma \ref{le:hard} to control $\left|\frac{\partial}{\partial t}\left(\frac{\widehat{\CI}_{t}(\widehat{w};\boldsymbol{d})}{t^4}\right)\right|$, the integral in \eqref{eq:CBLTc} equals
$$-\frac{1}{4}\CB_{L,\blambda}(\bc)\int_{T}^{2T}t^4\frac{\partial}{\partial t}\left(\frac{\widehat{\CI}_{t}(\widehat{w};\boldsymbol{d})}{t^4}\right)\operatorname{d}t+O_{\varepsilon}\left(\frac{T^{\varepsilon} L^{5+\varepsilon}}{|\bd|^{1-\varepsilon}}\right).$$ 
Therefore, on summing over the dyadic parameters $T\ll Q$ and all the $\bc$ with $F^*(\bc)=0$, the big $O$-terms above contribute 
\begin{align*}
	&\ll_{\varepsilon} B^{-\tau+\varepsilon}L^{5+\varepsilon}\sum_{\bc:\eqref{eq:latticecond},\eqref{eq:c00}\text{ hold}}|\bc|^{-1}\ll_{\varepsilon} B^{-\tau+\varepsilon}L^{5+\varepsilon}.
\end{align*}
 On the other hand, since $\widehat{\CI}_{t}(\widehat{w};\boldsymbol{d})$ vanishes for $t>\max (Q,2Q\sup_{\bt\in\operatorname{Supp}(\widehat{w})}|F(\bt)|)$ \cite[Lemma 4]{H-Bdelta}, and by Lemma \ref{le:hard}, $\lim_{t\to0}\widehat{\CI}_{t}(\widehat{w};\boldsymbol{d})=0$, we have
$$\int_{0}^{\infty}\frac{\partial}{\partial t}\widehat{\CI}_{t}(\widehat{w};\boldsymbol{d})\operatorname{d}t=0.$$
So we can further compute:
\begin{align*}
	\int_{0}^{\infty}t^4\frac{\partial}{\partial t}\left(\frac{\widehat{\CI}_{t}(\widehat{w};\boldsymbol{d})}{t^4}\right)\operatorname{d}t&=-4\int_{0}^{\infty}\frac{\widehat{\CI}_{t}(\widehat{w};\boldsymbol{d})}{t}\operatorname{d}t=-4\CJ_{L}(\widehat{w};\bc).
\end{align*} 
Finally for the variation term in \eqref{eq:CBLTc}, summing over $T$ gives $$\left[A_t\frac{\widehat{\CI}_{t}(\widehat{w};\boldsymbol{d})}{t^4}\right]_{1}^{\infty}\ll_\varepsilon B^{-\tau+\varepsilon}L^{5+\varepsilon},$$ by \eqref{eq:CBbd}, Proposition \ref{prop:qsumS} and Lemma \ref{le:hard}.

Hence we conclude
\begin{multline*}
	\underset{\substack{\bc\in\BZ^{4}\setminus\boldsymbol{0},F^*(\bc)=0\\q\ll Q}}{\sum\sum}\frac{S_{q,L,\blambda}(\bc)I_{q,L,\blambda}(\wtil_B;\bc)}{(qL)^4}\\ = \left(\frac{B}{L^2}\right)^4\left(\sum_{\substack{\bc: F^*(\bc)=0 \\\eqref{eq:latticecond},\eqref{eq:c00}\text{ hold}}}\CB_{L,\blambda}(\bc)\CJ_{L}(\widehat{w};\bc) +O_{\varepsilon}\left(B^{-\tau+\varepsilon}L^{5+\varepsilon}\right)\right).
\end{multline*} 
The error term above dominates the one arising in the contribution from $F^*(\bc)\neq 0$. On making use of Lemma \ref{le:easy} again, we can extend the range of the $\bc$-sum by removing the condition \eqref{eq:c00} upon adding an acceptable power-saving error term, and hence we extract the term $\CK_{L,\bGamma}(\widehat{w})$.
\end{proof}

\subsubsection{Proof of Theorem \ref{thm:Llambda}}
Going back to \eqref{eq:poisson}, by Proposition \ref{prop:c=0L} and \ref{prop:cneq0qsum} (the error term from the analysis of the case $\bc = \boldsymbol{0}$ in Proposition \ref{prop:c=0L} is dominated by the one in  Proposition \ref{prop:cneq0qsum}), we finish the proof of Theorem \ref{thm:Llambda}.
\qed

\subsection{From the affine cone to the projective quadric}\label{se:Vcount}

We begin with the following ``flipping lemma'' for $\CK_{L,d\bGamma}(\widehat{w})$ \eqref{eq:CKc}. 
\begin{lemma}\label{le:CKflip}
	For $(d,L)=1$, we have
	$$\CK_{L,d\bGamma}(\widehat{w})=\psi_F(d)\CK_{L,\bGamma}(\widehat{w}).$$
\end{lemma}
\begin{proof}
	Note that we may assume  $\varepsilon_{\Delta_F}\mid L$. Otherwise $\CB_{L,\blambda}(\bc)=0$ (recall \eqref{eq:CB}).
	Therefore, it follows directly from Lemma \ref{le:flip} that for every $\bc$, $$\CB_{L,d\blambda}(\bc)=\psi_F(d) \CB_{L,\blambda}(\bc).\qedhere$$ 
\end{proof}

Let us define the counting function for $\CW^o$:
\begin{equation}\label{eq:countWo}
	\CN_{\CW^o}(\wtil_B;(L,\bGamma)):=\sum_{\substack{\bx\in\BZ^4:F(\bx)=0\\\gcd(\bx)=1,\bx\equiv\bGamma\bmod L}}\wtil_B(\bx).
\end{equation}
\begin{lemma}\label{le:Wocount}
Assume that $L^2\asymp B^{1-\tau}$ for a certain $0<\tau<1$.
Then
	\begin{align*} \CN_{\CW^o}(\wtil_B;(L,\bGamma))=B^2\left(\CI(\widehat{w})\frac{\widetilde{\mathfrak{S}}_{L,\bGamma}(\CW)}{\BL(2,\chi_0[L])}+\CL(D_0)\CK_{L,\bGamma}(\widehat{w})\right)+O_\varepsilon\left(\frac{B^{2-\tau+\varepsilon}}{L} \right),
	\end{align*} where $D_0$ is defined by \eqref{eq:D0} and $$\CL(X):=\sum_{d\leqslant X,(d,L)=1}\frac{\mu(d)\psi_F(d)}{d^2}.$$
\end{lemma}
\begin{proof}
	A usual Möbius inversion yields $$\CN_{\CW^o}(\wtil_B;(L,\bGamma))=\sum_{d\ll B,(d,L)=1} \mu(d)\CN_{\CW}(\wtil_{B/d};(L,\overline{d}\bGamma)).$$
	Let $\tau_d'\in\BR$ be such that $$L^2=\left(\frac{B}{d}\right)^{1-\tau_d'}.$$
    For $d\leq D_0$ we then have that $-1+\varepsilon'<\tau_{D_0}'\leq \tau$ (for some $\varepsilon '>0$). An application of Theorem \ref{thm:Llambda} leads to
	\begin{align*}
	&B^2\left(\sum_{\substack{(d,L)=1\\ d\leqslant D_0}}\frac{\mu(d)}{d^2}\left(\CI(\widehat{w})\widetilde{\mathfrak{S}}_{L,\overline{d}\bGamma}(\CW)+\CK_{L,\overline{d}\bGamma}(\widehat{w})\right)\right)+	O_{\varepsilon}\left(\sum_{d\leq D_0}\frac{B^{2-\tau_d'+\varepsilon}}{d^{2-\tau_d'}L}\right).
	\end{align*}
For every $(d,L)=1$, we have $\sigma_p(\CW;L,\bGamma)=\sigma_p(\CW;L,d\bGamma)$  and hence $$\widetilde{\mathfrak{S}}_{L,\overline{d}\bGamma}(\CW)=\widetilde{\mathfrak{S}}_{L,\bGamma}(\CW).$$ 
Moreover, by Lemma \ref{le:CKflip}, 	$$\CK_{L,\overline{d}\bGamma}(\widehat{w})=\psi_F(d)\CK_{L,\bGamma}(\widehat{w}).$$
We observe that for $d\leq D_0$ we have 
$$\left(\frac{B}{d}\right)^{-\tau_d'}\ll d B^{-\tau}.$$
Noticing that 
$$\sum_{(d,L)=1}\frac{\mu(d)}{d^2}=\BL(2,\chi_0[L])^{-1},$$ as in the proof of Theorem \ref{thm:realzoommain}, the contribution of such $d\leqslant D_0$ is
$$B^2\left(\CI(\widehat{w})\frac{\widetilde{\mathfrak{S}}_{L,\bGamma}(\CW)}{\BL(2,\chi_0[L])}+\CL(D_0)\CK_{L,\bGamma}(\widehat{w})\right)+O_\varepsilon\left( \frac{B^{2-\tau+\varepsilon}}{L}\right).$$
Now for $D_0\leq d\ll B$, we use lattice point counting instead to show that
$$\CN_{\CW}(\wtil_{B/d};(L,\overline{d}\bGamma))\ll \#\{\bx\in\BZ^4:|\bx|\ll B/d,L\mid \bx\}\ll 1+\left(\frac{B}{dL}\right)^4,$$ whence the sum over $d>D_0$ is
$$\ll \sum_{d\ll B}1+\frac{B^4}{L^4}\sum_{d>D_0}\frac{1}{d^4}\ll_\varepsilon B+
\frac{B^{2-\tau+\varepsilon}}{L}. \qedhere$$ 
\end{proof}

We now make the additional assumption that the weight function $\wtil_B$ is symmetric, i.e. $\wtil_B (\bx)=\wtil_B (-\bx)$ for all $\bx\in \mathbb{R}^4$. Recall the definition of $\CN_{V}(\wtil_B;(L,\bLambda))$ in \eqref{eq:countingfunVpadic}.

\begin{lemma}\label{prop:Vcount}
Assume that $\wtil_B (\bx)=\wtil_B (-\bx)$ for all $\bx\in \mathbb{R}^4$ and that $L^2\asymp B^{1-\tau}$ for a certain $0<\tau<1$. We have
$$ \CN_{V}(\wtil_B;(L,\bLambda))=\frac{1}{2}B^2\CI(\widehat{w})\mathfrak{S}_{L,\bLambda}(\CV)+O_\varepsilon\left(B^{2-\tau+\varepsilon} \right).$$
\end{lemma}
\begin{proof}
	The function $\CN_{V}(\wtil_B;(L,\bLambda))$ is related to \eqref{eq:countWo} via
	\begin{align*}
		\CN_{V}(\wtil_B;(L,\bLambda))&=\frac{1}{2}\sum_{\gamma\in (\BZ/L\BZ)^\times}\CN_{\CW^o}(\wtil_B;(L,\gamma\bGamma)).
	\end{align*}
	Now for any $\gamma\in (\BZ/L\BZ)^\times$, by Lemma \ref{le:CKflip},  $$\widetilde{\mathfrak{S}}_{L,\gamma\bGamma}(\CW)=\widetilde{\mathfrak{S}}_{L,\bGamma}(\CW),\quad \CK_{L,\gamma\bGamma}(\widehat{w})=\psi_F(\gamma)\CK_{L,\bGamma}(\widehat{w}).$$
 Hence by Lemma \ref{le:Wocount}, $\CN_{V}(\wtil_B;(L,\bLambda))$ equals
\begin{multline*}
	\frac{1}{2} B^2\left(\CI(\widehat{w})\frac{\#(\BZ/L\BZ)^\times}{\BL(2,\chi_0[L])} \widetilde{\mathfrak{S}}_{L,\bGamma}(\CW)+\CL(D_0)\left(\sum_{\gamma\in (\BZ/L\BZ)^\times}\psi_F(\gamma)\right)\CK_{L,\bGamma}(\widehat{w})\right)\\+O_\varepsilon\left(B^{2-\tau+\varepsilon}\right).
\end{multline*}
Notice that $\sum_{\gamma\in (\BZ/L\BZ)^\times}\psi_F(\gamma)=0$ if  $\varepsilon_{\Delta_F}\mid L$. Otherwise $\CK_{L,\bGamma}(\widehat{w})=0$. So we conclude the desired formula on applying Proposition \ref{prop:TamagawaLLambda}.
\end{proof}

\begin{proof}[Proof of Theorem \ref{thm:padiczoommain}]	
	From Lemma \ref{prop:Vcount} we deduce the main term of Theorem \ref{thm:padiczoommain}. 
 
 It follows from Proposition \ref{prop:TamagawaLLambda} that $\mathfrak{S}_{L,\bLambda}(\CV)\gg_\varepsilon L^{-2-\varepsilon}$. Therefore, for the error term from Lemma \ref{prop:Vcount} to be negligible, we require that 
 $$ B^{-\tau + \varepsilon} \ll L^{-2 - \varepsilon}.$$
Recalling that $L\ll B^{\frac{1-\tau}{2}}$, this is satisfied if $1-\tau < \tau$, or equivalently, 
$$ \tau > \frac{1}{2}.\qedhere$$
\end{proof}

\section{Counting integral points on the punctured affine cone}\label{se:BMCK}
Let $W^o=W\setminus\boldsymbol{0}\subseteq \BA_{\BQ}^4\setminus\boldsymbol{0}$ be the ($3$-dimensional) punctured affine cone of the projective quadratic surface $V$. 
The Brauer--Manin obstruction to the Hasse principle and strong approximation of rational points has been studied by Colliot-Thélène and Xu \cite{CT-XuCompositio,CT-Xu}.
We refer to \cite{CT-XuCompositio}, \cite[\S2]{CT-Xu} and \cite[\S8.2]{Poonen} for more details about basic terminology. 

Let $W^o(\mathbf{A}_\BQ)$ be the adelic space, that is, the restricted product of $W^o(\BQ_\nu)$ with respect to $\CW^o(\BZ_\nu)$ for all places $\nu$ of $\BQ$. Our assumption that $V$ is everywhere locally soluble is equivalent to $W^o(\mathbf{A}_\BQ)\neq\varnothing$, which by the Hasse--Minkowski theorem, is also equivalent to $W^o(\BQ)\neq\varnothing$. Let $W^o(\mathbf{A}_\BQ)^{\operatorname{Br}}\subseteq W^o(\mathbf{A}_\BQ)$ denote the \emph{Brauer--Manin set}, consisting of elements that are orthogonal to the Brauer group.  We define $W^o(\mathbf{A}_\BQ^\infty)^{\operatorname{Br}}$, the \emph{Brauer--Manin set away from $\infty$}, to be the image of $W^o(\mathbf{A}_\BQ)^{\operatorname{Br}}$ under the projection $\operatorname{pr}_\infty:W^o(\mathbf{A}_\BQ)\to W^o(\mathbf{A}_\BQ^\infty)$, the adelic space without the $\BR$-component. The diagonal map $W^o(\BQ)\to W^o(\mathbf{A}_\BQ)$ then followed by $\operatorname{pr}_\infty$ maps $W^o(\BQ)$ into $W^o(\mathbf{A}_\BQ^\infty)^{\operatorname{Br}}$.
\begin{theorem}[Colliot-Thélène--Xu, \cite{CT-Xu} Theorem 6.5 (iii)]\label{thm:BrW}
	Assume that $\Delta_F\notin \BQ^2$ and that $W^o(\mathbf{A}_\BQ)\neq \varnothing$. Then the Brauer--Manin obstruction to strong approximation away from $\infty$ is the only one for the rational points of $W^o$, that is, $W^o(\BQ)$ is dense in $W^o(\mathbf{A}_\BQ^\infty)^{\operatorname{Br}}$. 
\end{theorem}
\begin{remark}
 Since $W^o(\BR)\neq\varnothing$ by assumption,  the form $F$ has signature either $(3,1)$ or $(2,2)$ over $\BR$. So upon diagonalising $F$ over $\BQ$, we may assume that $W^o$ is defined by the equation \begin{equation}\label{eq:quadricfibration}
		q(x,y,z)=p(t),
	\end{equation} where $q$ is an indefinite non-degenerate ternary quadratic form and $p(t)=ct^2,c\in\BQ^\times$. The assumption $\Delta_F\notin\BQ^2$ is therefore equivalent to $-c\det(q)\notin \BQ^2$ and hence \cite[Theorem 6.5 (iii)]{CT-Xu} applies. In fact \cite[\S6]{CT-Xu} deals with a more general family of quasi-affine varieties where $p(t)$ can be any non-zero polynomial in $t$, with the removal of the points $(0,0,0,t_0)$ where $t_0$ ranges over all multiple roots of $p(t)$. The special case \eqref{eq:quadricfibration} has been recently revisited again by work of Bright--Kok \cite{Bright-Kok} and Tronto \cite{Tronto}. 
	The proof of Theorem \ref{thm:BrW} proceeds by considering the quadratic fibration \eqref{eq:quadricfibration} into the variable $t$ (see \cite[Proposition 4.5]{CT-Xu}) and applies the machinery developed in  \cite[\S5]{CT-XuCompositio}. As pointed out in \cite[comment below Theorem 1.1]{Bright-Kok}, one cannot use simply the canonical projection $W^o\to V$ which is a $\BG_{\operatorname{m}}$-fibration.
\end{remark}

\begin{proposition}[\cite{CT-Xu} Proposition 5.7, \cite{Bright-Kok} Lemma 3.1]\label{prop:BrWo}
	Assume that $\Delta_F\notin \BQ^2$ and that $W^o(\mathbf{A}_\BQ)\neq \varnothing$.  Then we have $\operatorname{Br}(W^o)/\operatorname{Br}(\BQ)=\BZ/2\BZ$. Let  $g=g(X_1,\cdots,X_4)$ be a linear form that defines the tangent plane of a certain $\BQ$-point of $W^o$. Then the quaternion algebra $(\Delta_F,g)$ generates the group above.
\end{proposition}
 In the rest of this section, we choose the integral model $\CW^o:=\CW\setminus\overline{\boldsymbol{0}}\subseteq \BA_{\BZ}^4$ defined by $F$, where $\overline{\boldsymbol{0}}$ is the Zariski closure of $\boldsymbol{0}$ in $\BA_{\BZ}^4$. 
For $L\in\BN$ and $\bGamma\in\CW^o(\BZ/L\BZ)$, let us consider (as in \cite{PART1}) the \emph{congruence neighbourhood} (of $\bGamma$ modulo $L$)
$$\CE_{\CW^o}(L,\bGamma):= \prod_{p\mid L}\CE_p(L,\bGamma)\times\prod_{p\nmid L}\CW^o(\BZ_p)\subseteq \CW^o(\widehat{\BZ}),$$ where for every $p\mid L$, $$\CE_p(L,\bGamma):=\{z_p\in\CW^o(\BZ_p): z_p\equiv \bGamma\bmod p^{\operatorname{ord}_p(L)}\}.$$
The family of all congruence neighbourhoods $\left\{\CE_{\CW^o}(L,\bGamma)\right\}_{(L,\bGamma)}$ forms a topological basis for $\CW^o(\widehat{\BZ})$.

Let us fix $g$ an integral linear form defining the tangent plane at a certain integral point of $\CW^o$. If we assume that $W^0(\mathbb{Q})\neq \emptyset$ and $\Delta_F\notin\mathbb{Q}^2$, then by Proposition \ref{prop:BrWo}, the element $(\Delta_F,g)$ has order two and generates the Brauer group $\operatorname{Br}(W^o)/\operatorname{Br}(\BQ)$. Let \begin{equation}\label{eq:evarhof}
	\varrho_f:=\sum_{p<\infty}\operatorname{inv}_p(\Delta_F,g(\cdot)):W^o(\mathbf{A}_\BQ^\infty)\to \frac{1}{2}\BZ/\BZ
\end{equation} be the evaluation map, which is locally constant (see \cite[Corollary 8.2.11 (a)]{Poonen})
Restricting $\varrho_f$ to the compact space $\CW^o(\widehat{\BZ})$, we can choose $L_0\in\BN$ large enough such that $\varrho_f$ is constant on every congruence neighbourhood modulo $L_0$. Upon enlarging $L_0$, we may assume that  $2\Delta_F \mid L_0$ and that $\CW^o\bmod L_0$ is regular.

Let us define the normalised Tagamawa measure $\omega_f$ on $\CW^o(\widehat{\BZ})$ as follows. 
For every non-empty compact open $\CE\subseteq \CW^o(\widehat{\BZ})$ and for every $L$ with $L_0\mid L$, we cover $\CE$ by a finite number of disjoint congruence neighbourhoods modulo $L$: 
$$\CE=\bigsqcup_{\bGamma}\CE_{\CW^o}(L,\bGamma).$$
Then we let \begin{equation}\label{eq:omegaf}
	\omega_f(\CE):=\frac{1}{\BL(2,\chi_0[L])}\sum_{\bGamma}\widetilde{\mathfrak{S}}_{L,\bGamma}(\CW),
\end{equation} where $\widetilde{\mathfrak{S}}_{L,\bGamma}(\CW)$ is defined by \eqref{eq:singser}.

\begin{lemma}\label{lemcomp}
Assume that $W^o(\mathbb{R})\neq\varnothing$. If $\Delta_F>0$, then the real locus $W^o(\mathbb{R})$ consists of one connected component, and if $\Delta_F<0$, then $W^o(\mathbb{R})$ has two connected components over $\mathbb{R}$.
\end{lemma}

\begin{proof}
As $W^o(\mathbb{R})\neq \emptyset$, we may assume after a change of variables that $W^o$ is the punctured affine cone of the non-singular quadratic form
$$t_0t_1=F_2(t_2,t_3)$$
where $F_2(t_2,t_3)$ is a binary quadratic form over $\mathbb{R}$. In particular, if $\Delta_F>0$, then $F_2(t_2,t_3)$ has a non-trivial zero over $\mathbb{R}$, and if $\Delta_F<0$, then $F_2(t_2,t_3)$ is a positive or negative definite quadratic form over $\mathbb{R}$. I.e. after a change of variables, we may assume that $W^o$ is the punctured affine cone of $t_0t_1=t_2t_3$ in the case $\Delta_F>0$ or is given by the punctured affine cone of $t_0t_1=t_2^2+t_3^2$ if $\Delta_F<0$. The real locus of the first one is connected, whereas the second punctured affine cone has two connected components over $\mathbb{R}$, which are split by the Brauer class $(t_1,-1)=(t_0,-1)$.
\end{proof}

If the real locus $W^o(\BR)$ has two connected components, then we write $W_1,W_2$ for these. Note that we have
\begin{align*}
	\operatorname{inv}_\infty(\Delta_F,g)|_{W_1}=\operatorname{inv}_\infty(\Delta_F,g)|_{W_2}=0\quad & \text{ if } \Delta_F>0;\\ 
	\operatorname{inv}_\infty(\Delta_F,g)|_{W_1}=\operatorname{inv}_\infty(\Delta_F,g)|_{W_2}+\frac{1}{2}\quad & \text{ if } \Delta_F<0. 
\end{align*}

Throughout this section, $L$ is considered fixed (i.e. we take $\tau=1$). We omit its dependency in all implied constants. Let $\wtil_B$ be defined by \eqref{eq:weightpm}. 
We shall prove
\begin{theorem}\label{thm:Wocount}
	Assume $L_0\mid L$ and let $\bGamma_0\in\CW^o(\BZ/L\BZ)$. In the case $\Delta_F<0$, assume that the support  of $w_0$ is contained in one of the components $W_j$ for some $1\leq j\leq 2$, and that $W_j\times\CE_{\CW^o}(L,\bGamma_0)\subseteq W^o(\mathbf{A}_\BQ)^{\operatorname{Br}}$.  Then 
	$$\CN_{\CW^o}(\wtil_B;(L,\bGamma_0))= 2\CI(\widehat{w})\omega_f\left(\CE_{\CW^o}(L,\bGamma_0)\right) B^2+O_\varepsilon(B^{1+\varepsilon}),$$ where $\CI(\widehat{w})$ is defined by \eqref{eq:weightedsingintpm}. 
\end{theorem}
Theorem \ref{thm:Wocount} offers a new proof and a quantitative version of strong approximation with Brauer--Manin obstruction for integral points on $\CW^o$, i.e. Theorem \ref{thm:BrW}. On combining Lemmas \ref{le:CKflip}, \ref{le:Wocount} and Theorem \ref{thm:Wocount}, we arrive at determining the  term $\CK_{L,\bGamma}(\widehat{w})$, and we can deduce cleaner asymptotic formulas for $\CN_{\CW}(\wtil_B;(L,\bGamma))$ in Theorem \ref{thm:Llambda} (without the term $\CK_{L,\bGamma}(\widehat{w})$). This completes the result in \cite[\S5]{Lindqvist}.
\begin{corollary}\label{co:CK}
	Under the assumptions in Theorem \ref{thm:Wocount}, we have
	$$\CK_{L,\bGamma_0}(\widehat{w})=\CI(\widehat{w})	\omega_f\left(\CE_{\CW^o}(L,\bGamma_0)\right)\BL(2,\psi_F\chi_0[L]).$$ Moreover, for every $\gamma\in (\BZ/L\BZ)^\times$, we have
	$$\CN_{\CW}(\wtil_B;(L,\gamma\bGamma_0))=B^2\CI(w)\widetilde{\mathfrak{S}}_{L,\bGamma_0}(\CW)\left(1+\psi_F(\gamma)\frac{\BL(2,\psi_F\chi_0[L])}{\BL(2,\chi_0[L])}\right)+O_\varepsilon(B^{1+\varepsilon}).$$ 
\end{corollary}
\begin{proof}
	We start by computing the term $\CL(D_0)$ introduced in Lemma \ref{le:Wocount}:
	$$\sum_{d\leqslant D_0,(d,L)=1}\frac{\mu(d)\psi_F(d)}{d^2}=\BL(2,\psi_F\chi_0[L])^{-1}+O_\varepsilon(B^{-1+\varepsilon}).$$
	The formula for $\CK_{L,\bGamma_0}(\widehat{w})$ follows directly from comparing Lemma \ref{le:Wocount} and Theorem \ref{thm:Wocount}.
	
	To deduce the formula for $\CN_{\CW}(\wtil_B;(L,\gamma\bGamma_0))$ it suffices to combine Theorem \ref{thm:Llambda} with the formula for $\CK_{L,\bGamma_0}(\widehat{w})$ above, and then to use Lemma \ref{le:CKflip}.
\end{proof}
\begin{remark}\label{rmk:measureint}
	Assuming $\Delta_F<0$ and that the support of $w_0$ is contained in  $W_1$ and assuming that $(W_1\times\CE_{\CW^o}(L,\bGamma_0))\cap W^o(\mathbf{A}_\BQ)^{\operatorname{Br}}=\varnothing$, we similarly deduce
	$$\CK_{L,\bGamma_0}(\widehat{w})=-\CI(\widehat{w})	\omega_f\left(\CE_{\CW^o}(L,\bGamma_0)\right)\BL(2,\psi_F\chi_0[L])$$ and the corresponding formula for $\CN_{\CW}(\wtil_B;(L,\gamma\bGamma_0))$. From these formulas we see that the value of $\CK$ is  ``biased'' depending on whether the adelic sets meet the Brauer--Manin locus or not.
\end{remark}

Before preceding to the proof of Theorem \ref{thm:Wocount}, we make the following key observation:
\begin{lemma}\label{le:gammaBr}
Assume $L_0\mid L$ and let $\bGamma_0\in\CW^o(\BZ/L\BZ)$. In the case $\Delta_F<0$, assume that the support  of $w_0$ is contained in one of the components $W_j$ for $1\leq j\leq 2$ and that $W_j\times\CE_{\CW^o}(L,\bGamma_0)\subseteq W^o(\mathbf{A}_\BQ)^{\operatorname{Br}}$.
	In the case $\Delta_F<0$, for every $\gamma\in(\BZ/L\BZ)^\times$, we have
	\begin{align*}
		W_j\times \CE_{\CW^o}(L,\gamma\bGamma_0)\subseteq W^o(\mathbf{A}_\BQ)^{\operatorname{Br}}&\Leftrightarrow \psi_F(\gamma)=1;\\ W_j\times\CE_{\CW^o}(L,\gamma\bGamma_0)\cap W^o(\mathbf{A}_\BQ)^{\operatorname{Br}}=\varnothing&\Leftrightarrow \psi_F(\gamma)=-1.
	\end{align*}
    In the case $\Delta_F>0$ the same conclusions hold with $W_j$ replaced by $W^o(\mathbb{R})$.
\end{lemma}
\begin{proof}
		By Hensel's lemma, for every $p\mid L$, we can lift $\bGamma_0$ into $\CW^o(\BZ_p)$ and we denote such a lift by $(\bGamma_0)_p$. By assumption, we have
\begin{equation}\label{eq:Br1}
			\operatorname{inv}_\infty(\Delta_F,g)|_{W_j}+\sum_{p\nmid L} \operatorname{inv}_p(\Delta_F,g)|_{\CW^o(\BZ_p)}+\sum_{p\mid L}\operatorname{inv}_p(\Delta_F,g((\bGamma_0)_p))=0.
\end{equation}
		Now for every $\gamma\in(\BZ/L\BZ)^\times$, we then have $W_j\times \CE_{\CW^o}(L,\gamma\bGamma_0)\subseteq W^o(\mathbf{A}_\BQ)^{\operatorname{Br}}$ if and only if
\begin{equation}\label{eq:Br2}
				\operatorname{inv}_\infty(\Delta_F,g)|_{W_1}+\sum_{p\nmid L} \operatorname{inv}_p(\Delta_F,g)|_{\CW^o(\BZ_p)}+\sum_{p\mid L}\operatorname{inv}_p(\Delta_F,g(\gamma(\bGamma_0)_p))=0.
\end{equation} 
On the other hand, 
$W_j\times \CE_{\CW^o}(L,\gamma\bGamma_0)\cap W^o(\mathbf{A}_\BQ)^{\operatorname{Br}}=\varnothing$ if and only if the right-hand-side of \eqref{eq:Br2} is equal to $\frac{1}{2}$. 
Since $g$ is homogeneous, we have $$\sum_{p\mid L}\operatorname{inv}_p(\Delta_F,g(\gamma(\bGamma_0)_p))=\sum_{p\mid L}\operatorname{inv}_p(\Delta_F,g((\bGamma_0)_p))+\sum_{p\mid L}\operatorname{inv}_p(\Delta_F,\gamma).$$ Given  \eqref{eq:Br1}, the condition \eqref{eq:Br2}  holds (resp. does not hold) if and only if $$\sum_{p\mid L}\operatorname{inv}_p(\Delta_F,\gamma)=0 \quad \left(\text{resp. }\sum_{p\mid L}\operatorname{inv}_p(\Delta_F,\gamma)=\frac{1}{2}\right).$$ 
We choose a representative $\gamma\in \mathbb{Z}$ which is a positive integer, for the residue class $\gamma\in (\mathbb{Z}/L\mathbb{Z})^\times$. Note that  $\operatorname{inv}_p(\Delta_F,\gamma)=0$ for all $p\nmid L\gamma$ by \cite[Corollary 6.9.3]{Poonen} (see also \cite[III.1 Theorem 1]{Serre}). 
By Hilbert's reciprocity law (see \cite[III.2 Theorem 3]{Serre}), we have \begin{equation}\label{eq:hilbert}
    \sum_{p\mid L}\operatorname{inv}_p(\Delta_F,\gamma)=\sum_{p\mid \gamma}\operatorname{inv}_p(\Delta_F,\gamma). 
\end{equation}Since $$\psi_F(\gamma)=\prod_{p\mid \gamma}\left(\frac{\Delta_F}{p}\right)^{\operatorname{ord}_p(\gamma)},$$ (recall that $2\mid L$ and $2\nmid \gamma$) a usual computation of the Hilbert symbol (see \cite[III.1 Theorem 1]{Serre}) shows that 
$$\sum_{p\mid \gamma}\operatorname{inv}_p(\Delta_F,\gamma)=\begin{cases}
	0&\text{ if } \psi_F(\gamma)=1;\\
	\frac{1}{2}&\text{ if } \psi_F(\gamma)=-1.\\
\end{cases}$$ 	
Recalling \eqref{eq:hilbert}, this finishes the proof of Lemma \ref{le:gammaBr}.
\end{proof}

\begin{remark}
	Lemma \ref{le:gammaBr} also recovers the main result of \cite{Bright-Kok} (see also the last remark in \cite{Bright}), where a special form of $F$ defining $W^o$ is investigated in detail, and it is shown that at most half of the multiples of a fixed primitive solution modulo a certain prime can be lifted into integers points of $\CW^o$.
\end{remark}
\begin{proof}[Proof of Theorem \ref{thm:Wocount}]
	We take $\tau=1$ in Theorem \ref{thm:Llambda} so that the error term is $O_\varepsilon(B^{-1+\varepsilon})$. By Lemma \ref{le:gammaBr}, for every $\gamma\in(\BZ/L\BZ)^\times$ with $\psi_F(\gamma)=-1$, one has $	W_j\times\CE_{\CW^o}(L,\gamma\bGamma_0)\cap\CW^o(\BZ)=\varnothing$, so $$\CN_{\CW^o}(\wtil_B;(L,\gamma\bGamma_0))=0.$$
	On the other hand, by Lemmas \ref{le:CKflip} and \ref{le:Wocount} we have
\begin{align*} \CN_{\CW^o}(\wtil_B;(L,\gamma\bGamma_0))=B^2\left(\CI(\widehat{w})\frac{\widetilde{\mathfrak{S}}_{L,\bGamma_0}(\CW)}{\BL(2,\chi_0[L])}+\CL(D_0)\psi_F(\gamma)\CK_{L,\bGamma_0}(\widehat{w})\right)+O_\varepsilon\left(B^{1+\varepsilon}\right),
	\end{align*}

We recall that
$$\CL(D_0)=\BL(2,\psi_F\chi_0[L])^{-1}+O_\varepsilon(B^{-1+\varepsilon}),$$
and hence we obtain
$$\BL(2,\psi_F\chi_0[L])^{-1}\CK_{L,\bGamma_0}(\widehat{w})= \CI(\widehat{w})\frac{\widetilde{\mathfrak{S}}_{L,\bGamma_0}(\CW)}{\BL(2,\chi_0[L])}$$

Again by Lemma \ref{le:Wocount} we deduce that

    $$\CN_{\CW^o}(\wtil_B;(L,\bGamma_0))= 2\CI(\widehat{w})\omega_f\left(\CE_{\CW^o}(L,\bGamma_0)\right) B^2+O_\varepsilon(B^{1+\varepsilon}).\qedhere$$
    
\end{proof}
\subsection{Proof of Theorem \ref{thm:HLWo}}
We recall \eqref{eq:evarhof} and define in the case of $\Delta_F<0$ the functions $$\Xi_{W_i}:=2\mathds{1}_{\varrho_f+\operatorname{inv}_\infty(\Delta_F,g)|_{W_i}=0},\quad i=1,2.$$ 
In the case $\Delta_F>0$ we define 
$$\Xi:=2\mathds{1}_{\varrho_f=0}.$$ 
The result then follows from Theorem \ref{thm:Wocount} by considering neighborhoods $\CE$ of the form  $\CE_{\CW^o}(L,\bGamma_0)$ with $L_0\mid L$. 



\qed
\begin{remark}\label{rmk:noCK}
	According to \cite[p. 159--p. 160]{H-Bdelta} we have \begin{equation}\label{eq:globalcount}
	    \sum_{\bx\in \CW^o(\BZ)} w_B(\bx)\sim \CI(w)\omega_f(\CW^o(\widehat{\BZ})) B^2.
	\end{equation}
	We now explain why in the global counting \eqref{eq:globalcount} such a complication was not observed. Continuing Remark \ref{rmk:deltanotdiv}, we can now easily see that the integral Brauer--Manin set $W^o(\BR)\times \CW^o(\widehat{\BZ})\cap W^o(\mathbf{A}_\BQ)^{\operatorname{Br}}$ has exactly half of the ``volume'' of $W^o(\BR)\times \CW^o(\widehat{\BZ})$. Indeed, assuming first that $\Delta_F>0$, we have $$\Xi_{W_i}= 2\mathds{1}_{\varrho_f=0},\quad i=1,2.$$ Decompose  $$\CW^o(\widehat{\BZ})=\bigsqcup_{\bGamma}\bigsqcup_{\gamma\in (\BZ/L\BZ)^\times}\CE_{\CW^o}(L,\gamma\bGamma)$$ for any $L$ with $ L_0\mid L$ and a family of primitive representatives $\bGamma\bmod L$ which cover $\CV(\BZ/L\BZ)$. The character $\psi_F$ achieves $1$ (resp. $-1$) for exactly half of the elements of $(\BZ/L\BZ)^\times$, which determines completely whether the congruence neighbourhood $\CE_{\CW^o}(L,\gamma\bGamma)$ meets the Brauer--Manin set or not for every such fixed $\bGamma$. 
	Hence only half of the Tamagawa volume of $\CW^o(\widehat{\BZ})$ counts and the factors $2$ compensates such lost.
	As for the case $\Delta_F<0$, for each  fixed $\CE_{\CW^o}(L,\bGamma)$ only one of the two connected component meets the Brauer--Manin set. Hence only half of the ``volume'' of $W^o(\BR)$ counts, which is again compensated by the factor $2$.
\end{remark}

\section*{Acknowledgements}
Part of this work was done at IST Austria, at Georg-August-Universität Göttingen and at Southern University of Science and Technology of China. Their excellent working condition and financial support are greatly appreciated. We thank Yang Cao and Daniel Loughran for helpful discussions. We are very grateful to Roger Heath-Brown for comments on an earlier version of this article. The third author was supported by the University of Bristol and the Heilbronn Institute for Mathematical Research.
	
\end{document}